\documentclass[reqno]{amsart}
\usepackage{amsthm,amsmath,amssymb}
\usepackage{stmaryrd,mathrsfs}
\usepackage{esint}
\usepackage{tikz}
\usepackage{hyperref}
\usepackage{enumerate}

\allowdisplaybreaks

\newcommand{\ud}[0]{\,\mathrm{d}}

\newcommand{\abs}[1]{|#1|}
\newcommand{\Babs}[1]{\Big|#1\Big|}

\newcommand{\Norm}[2]{\|#1\|_{#2}}

\newcommand{\BNorm}[2]{\Big\|#1\Big\|_{#2}}

\newcommand{\ave}[1]{\langle #1\rangle}

\newcommand{\loc}[0]{\operatorname{loc}}

\newcommand{\R}{\mathbb{R}}

\newcommand{\C}{\mathbb{C}}

\newcommand{\N}{\mathbb{N}}
\newcommand{\Z}{\mathbb{Z}}

\newcommand{\eps}[0]{\varepsilon}

\newcommand{\Lip}[0]{\operatorname{Lip}}
\newcommand{\lip}[0]{\operatorname{lip}}

\swapnumbers \numberwithin{equation}{section}

\theoremstyle{plain}
\newtheorem{theorem}[equation]{Theorem}
\newtheorem{proposition}[equation]{Proposition}
\newtheorem{corollary}[equation]{Corollary}
\newtheorem{lemma}[equation]{Lemma}

\theoremstyle{definition}
\newtheorem{definition}[equation]{Definition}

\theoremstyle{remark}
\newtheorem{remark}[equation]{Remark}
\newtheorem{example}[equation]{Example}

\makeatletter
\@namedef{subjclassname@2010}{%
  \textup{2010} Mathematics Subject Classification}
\makeatother

\begin{document}

\title[Characterisations of Sobolev spaces]{Characterisations of Sobolev spaces and constant functions over metric spaces}


\author[Tuomas Hyt\"onen \and Riikka Korte]{Tuomas Hyt\"onen \and Riikka Korte}
\address{Aalto University,
Department of Mathematics and Systems Analysis,
P.O. Box 11100, FI-00076 Aalto, Finland}
\email{tuomas.hytonen@aalto.fi}
\email{riikka.korte@aalto.fi}

\thanks{The authors were supported by the Research Council of Finland through projects 346314 (T.H.), 360184 (R.K.), 364208 (T.H.), and 371637 (T.H.).}





\begin{abstract}
In a doubling metric measure space $(X,\rho,\mu)$ supporting a Poincar\'e inequality, we give a new characterisation of first-order Sobolev spaces by mean oscillations, and extend previous characterisations of constant functions in terms of the finiteness of certain integrals through a new approach. As a key tool of independent potential, we introduce a novel ``macroscopic'' Poincar\'e inequality, whose right-hand side has oscillations of the same form as the left-hand side, but at a smaller macroscopic scale $r\in(0,R)$.

Besides intrinsic interest, these results are motivated by applications to quantitative compactness properties of commutators $[f,T]$ of pointwise multipliers and singular integrals. With pivotal use of the present results, a characterisation of commutator mapping properties, over the same class of general domains $(X,\rho,\mu)$, is obtained in a companion paper.
\end{abstract}

\maketitle


\tableofcontents

\section{Introduction}

This investigation arose from a desire to obtain a common framework for an array of recent results dealing with Schatten class $S^p$ properties of commutators
\begin{equation*}
  \phi\mapsto [f,T]\phi:=f T(\phi)-T(f\phi)
\end{equation*}
of pointwise multipliers $f$ and singular integral operators $T$ in a variety of different settings \cite{FLL:23,FLLX,FLMSZ,GLW:23,LXY} beyond the unweighted Euclidean space $\R^d$, where such results go back to \cite{JW:82,RS:NWO}, as well as \cite[Appendix]{CST} in a critical borderline case.

The proofs of such results naturally split into two parts of different spirit:
\begin{enumerate}[\rm(a)]
  \item\label{it:Sp=Osc} An interplay of operators and functions, where
  the Schatten norms of the commutator $[f,T]$ are connected to abstract oscillatory norms of the function $f$.
  \item\label{it:Osc=concrete} Analysis of functions only, where these abstract norms are further characterised in terms of more concrete and familiar function space properties. 
\end{enumerate}
The aim of this paper is to provide the relevant results of type \eqref{it:Osc=concrete} that, combined with results of type \eqref{it:Sp=Osc} in the companion paper \cite{Hyt:Sp}, give a general characterisation of the Schatten class membership of commutators stated in \cite[Theorem 1.1]{Hyt:Sp}.

More specifically, in a general setting of a metric measure space $(X,\rho,\mu)$, our goal is to provide the essential function space theory underlying the following key phenomena in the Schatten class properties of commutators:
\begin{enumerate}[\rm(1)]
  \item\label{it:cut-off} The {\em cut-off phenomenon} of \cite{JW:82} says that, for a threshold value $d$, we have $[f,T]\in S^p$ for $p\leq d$ if and only if $f$ is constant (and hence $[f,T]=0$).
  \item\label{it:CST,RS} At the critical index $d$, the weak-type Schatten property $[f,T]\in S^{d,\infty}$ is characterised (combining results of \cite{CST,RS:NWO}) by a first-order Sobolev norm of $f$.
\end{enumerate}

Aside from this initial motivation and application, we believe that our results also have independent interest in the context of the following two parallel themes that have been extensively studied both in $\R^d$ and beyond:
\begin{enumerate}[\rm(1)]
  \item\label{it:recConst} Characterisation of constant functions by the finiteness of suitable integral expressions---a theme going back to \cite{Brezis:02} in $\R^d$ and subsequently extended to various settings in \cite{BS:18,BP:03,HKT,KSS,PP:04,RM:12}.
  \item\label{it:DfreeW} Derivative-free characterisations of Sobolev spaces, several versions of which have been explored even in $\R^d$ \cite{AMV:12,BBM:01,BN:06,BSVY,BVY,CST,Frank,LMSZ,Nguyen,Nguyen:08,Nguyen:25}, while extensions and variants over other spaces appear in \cite{BS:18,DLYY,DMS:19,HXZ,LPZ:JGA,LPZ:NA,LXY,Shimizu:25}.
\end{enumerate}

However, rather than just extending these lists, our present characterisations are specifically tailored for the needs of the commutator questions discussed above, and their usefulness is confirmed by their successful application in the companion work \cite{Hyt:Sp}. Our main theorem is the following generalisation of a recent Euclidean result of \cite{Frank} and its Carnot group extension \cite{LXY}; we will say more about the relations to earlier literature after the statement:

\begin{theorem}\label{thm:Rupert}
Let $(X,\rho,\mu)$ be a complete doubling metric measure space supporting a $(1,p)$-Poincar\'e inequality (Definition \ref{def:Poincare}) for some $p\in(1,\infty)$. Let $f\in L^1_{\loc}(\mu)$. Then $f$ belongs to the homogeneous Haj\l{}asz--Sobolev space $\dot M^{1,p}(\mu)$ (Definition \ref{def:Hajlasz}) if and only if $m_f\in L^{p,\infty}(\nu_p)$, where $m_f$ and $\nu_p$ are the function and the measure on $X\times(0,\infty)$ defined by
\begin{equation}\label{eq:mb-nu}
   m_f(x,t):=\fint_{B(x,t)}\abs{f-\ave{f}_{B(x,t)}}\ud\mu,\qquad
  \ud\nu_p(x,t):=\frac{\ud\mu(x)\ud t}{t^{p+1}}.
\end{equation}
In this case, we have
\begin{align}
  \Norm{f}{\dot M^{1,p}(\mu)} \label{eq:CST}
  &\sim\Norm{m_f}{L^{p,\infty}(\nu_p)}:=\sup_{\kappa>0}\kappa\cdot\nu_p(\{m_f>\kappa\})^{\frac 1p} \\
  &\sim\liminf_{\kappa\to0}\kappa\cdot\nu_p(\{m_f>\kappa\})^{\frac 1p}. \label{eq:Rupert}
\end{align}
\end{theorem}

We have formulated Theorem \ref{thm:Rupert} in terms of the Haj\l{}asz--Sobolev space $\dot M^{1,p}(\mu)$ of \cite{Hajlasz:96}, since this turns out to be the most convenient definition to work with for the present purposes. However, under the assumptions of Theorem \ref{thm:Rupert}, this space is known to be equivalent so several other established notions of Sobolev spaces over a metric measure space; see Remark \ref{rem:SobolevSpaces} for details.

The norm comparison \eqref{eq:CST} is an exact analogue of the corresponding results in $\R^d$ \cite[Theorem 1.1]{Frank} and Carnot groups \cite[Theorem A.3]{LXY}, but its proof will require substantial departures from \cite{Frank,LXY} in the lack of both group structure and higher-order smoothness. See Examples \ref{ex:HajlaszRd} and \ref{ex:Carnot} for the coincidence of $\dot M^{1,p}(\mu)$ with the usual Sobolev spaces in the respective settings. On the other hand, our limit relation \eqref{eq:Rupert} is slightly weaker than the corresponding identity
\begin{equation*}
  \Norm{f}{\dot W^{1,p}(\R^d)}=c_{d,p}\lim_{\kappa\to 0}\kappa\cdot\nu_p(\{ m_f>\kappa\})^{\frac1p}
\end{equation*}
proved in \cite[Eq.\ (1.2)]{Frank}, where $c_{d,p}$ is an explicit constant, and the existence of the limit is part of the result. In the lack of the symmetries available in $\R^d$, we believe that \eqref{eq:Rupert} is the best substitute that one can expect in our generality.

Another previous result in the spirit of Theorem \ref{thm:Rupert} is \cite[Theorem 1.4]{BS:18}: in addition to our assumptions, it requires the space to be compact and Ahlfors $d$-regular, and only deals with integrability $p=d$ matching the homogeneous dimension of the space---in the spirit of \cite{CST}, which imposed the same assumption on $\R^d$.

In contrast to the recent results of \cite{DLYY,HXZ} (which extend a different characterisation of the Sobolev space from \cite{BSVY,BVY}), we would like to stress that Theorem \ref{thm:Rupert} characterises the {\em whole space} $\dot M^{1,p}(\mu)$, rather than just giving an {\em equivalent norm of test functions} that are already known to be in the space. Indeed, some of the key challenges in the proof are about showing that the finiteness of the oscillatory norm $\Norm{m_f}{L^{p,\infty}(\nu_p)}$ ensures that $f\in\dot M^{1,p}(\mu)$, without  {\em a priori} knowing whether $f$ belongs to this space.

Concerning the characterisation of constants, our second main result is the following, where we denote $V(x,y):=\mu(B(x,\rho(x,y)))$ for all $x,y\in X$.

\begin{theorem}\label{thm:const-intro}
Let $(X,\rho,\mu)$ be a doubling metric measure space. For $p,s\in(0,\infty)$, consider the following two variants of homogeneous Besov spaces:
\begin{equation*}
\begin{split}
  \dot B^s_{p,p}(\mu) &:=\Big\{ f\in L^p_{\loc}(\mu): \iint_{X\times X}\Big(\frac{\abs{f(x)-f(y)}}{\rho(x,y)^s}\Big)^p\frac{\ud\mu(x)\ud\mu(y)}{V(x,y)}<\infty\Big\}, \\
  \dot B^p(\mu) &:=\Big\{ f\in L^p_{\loc}(\mu): \iint_{X\times X}\abs{f(x)-f(y)}^p\frac{\ud\mu(x)\ud\mu(y)}{V(x,y)^2}<\infty\Big\}.
\end{split}
\end{equation*}
Let $d\in[1,\infty)$, and suppose that $(X,\rho,\mu)$ satisfies a $(1,d)$-Poincar\'e inequality.
\begin{enumerate}[\rm(1)]
  \item\label{it:B1dd=const} Then $\dot B^1_{d,d}(\mu)=\{\operatorname{constants}\}$.
  \item\label{it:Bp=const} If, in addition, $(X,\rho,\mu)$ has lower dimension $d$ (Definition \ref{def:dims}), then
\begin{equation*}
   \dot B^p(\mu)=\{\operatorname{constants}\}\qquad\forall\ p\in(0,d].
\end{equation*}
\end{enumerate}
\end{theorem}

We would like to stress that Theorem \ref{thm:const-intro}, in contrast to Theorem \ref{thm:Rupert}, does not assume the completeness of $X$. This reflects the relatively simpler considerations that enter into its proof.

It seems that $\dot B^s_{p,p}(\mu)$ is the most relevant version of a Besov (or fractional Sobolev) space over a metric measure space for many purposes \cite{BB:23,BBS:22,GKS:24,GKS:10}, but the variant $\dot B^p(\mu)$ is one that naturally arises in the context of commutators \cite{FLLX,Hyt:Sp}. If $X$ is Ahlfors $d$-regular, thus $V(x,y)\sim\rho(x,y)^d$, then $\dot B^p(\mu)=\dot B^{d/p}_{p,p}(\mu)$---consistent with the role of $\dot B^{d/p}_{p,p}(\R^d)$ in classical commutator results of \cite{JW:82}. In this case, the condition $p\leq d$ means that $s=d/p\geq 1$, reflecting the fact that $\dot B^s_{p,p}$ above is not the right definition of the Besov space on $\R^d$ for smoothness $s\geq 1$.

 Case \eqref{it:B1dd=const} of Theorem \ref{thm:const-intro} is recently due to \cite[Corollary 1.8]{KSS}, but we provide a different approach of potentially independent interest, being relatively direct in using only the Besov spaces as defined above, while the proof of \cite[Corollary 1.8]{KSS}---quoting \cite[Theorem 5.1]{Baudoin:24}, \cite[Theorem 10.5.2]{HKST}, and \cite[Theorem 1.6]{KSS}---passes through the theory of Korevaar--Schoen spaces $KS^{1,p}$. In fact, we prove a slightly stronger local version of case \eqref{it:B1dd=const} in Proposition~\ref{prop:constBesov}.

On the other hand, case \eqref{it:Bp=const} is a new result to our knowledge, and it is this variant that is relevant for commutator applications, both in the companion paper \cite{Hyt:Sp} and in the recent \cite{FLLX}: In the specific ``Bessel setting'' $(X,\rho,\mu)=(\R_+^{n+1},\abs{x-y},x_{n+1}^{2\lambda}\ud x)$ with $d=n+1$, \cite[Proposition 6.4]{FLLX} obtain a version of case \eqref{it:Bp=const} in the weaker form
\begin{equation*}
  C^2(\R_+^{n+1})\cap\dot B^p(x_{n+1}^{2\lambda}\ud x)=\{\operatorname{constants}\}\qquad\forall\ p\in(0,n+1].
\end{equation*}
The fact that our Theorem \ref{thm:const-intro} leads to an improvement of this previous result, even when specialised to such a concrete situation, illustrates the power of our abstract methods. See Remark \ref{rem:Bessel} below for further elaboration on this. Again, we also obtain a slightly stronger local variant of case \eqref{it:Bp=const} in Corollary \ref{cor:constBesov}.

Both Theorems \ref{thm:Rupert} and \ref{thm:const-intro} depend on the (by now usual, cf.~\cite{Heinonen:book,HKST})  assumption of a Poincar\'e inequality, which we recall in Definition \ref{def:Poincare}. Roughly speaking, this inequality says that oscillations of a function on a macroscopic scale are controlled by ones on the infinitesimal scale. As a key intermediate step between this assumption and Theorems \ref{thm:Rupert} and \ref{thm:const-intro}, we establish a new ``macroscopic Poincar\'e inequality'': oscillations on a larger macroscopic scale are controlled by ones on a smaller macroscopic scale, uniformly over the two scales. Formulated in Theorem \ref{thm:macroPoincare}, this is our third main result, which should have independent interest and other applications elsewhere.


\subsection{Plan of the paper and an overview of the argument}

In the following Section \ref{sec:prelim}, we collect some preliminaries, including the definitions omitted in this Introduction.
In Section \ref{sec:macroPoincare}, we present our key new tool of the ``macroscopic'' Poincar\'e inequality in Theorem \ref{thm:macroPoincare}. With this result, we obtain a rather quick proof of Theorem \ref{thm:const-intro} on the characterisation of constant functions in Section \ref{sec:constants}.

The rest of the paper is dedicated to the proof of our Main Theorem \ref{thm:Rupert} on the characterisation of the Sobolev space. The rough outline of our proof is adapted from the Euclidean argument of \cite{Frank}. In particular, the lower bound of the Sobolev norm (i.e., ``$\gtrsim$'' in \eqref{eq:CST}) is obtained in Section~\ref{sec:MpLowerBd} by a straightforward adaptation of \cite{Frank}. However, the other direction is much more challenging, and we need to develop substantial new methods.
Besides the macroscopic Poincar\'e inequality, which also features in this proof, we establish some delicate comparison of maximal and mean oscillations at the infinitesimal scale in Section \ref{sec:Lip-gf}, to substitute tools like Taylor expansions, which are not available in our setting. However, we only obtain this comparison for Lipschitz functions. To lift this {\em a priori} assumption, we need to be able to approximate general functions by Lipschitz functions with respect to the ``exotic'' norm $f\mapsto\Norm{m_f}{L^{p,\infty}(\nu_p)}$ featuring in the desired characterisation. (This approximation follows easily from Theorem \ref{thm:Rupert} {\em a posteriori}, once we know that the exotic norm is equivalent to $\Norm{f}{\dot M^{1,p}(\mu)}$; but we need it {\em a priori}, in order to prove this equivalence!) This is accomplished in Section \ref{sec:LipApprox} with the help of the macroscopic Poincar\'e inequality, this time serving to substitute some convolution tricks, which are easy in~$\R^d$ \cite{Frank}, and also on Carnot groups \cite{LXY}, but unavailable in our setting. After these preparations, we can complete the proof of the upper bound of the Sobolev norm (``$\lesssim$'' in \eqref{eq:CST}) in Section~\ref{sec:MpUpperBd} and the limit relation \eqref{eq:Rupert} in Section~\ref{sec:limit}.

In the final Section \ref{sec:final}, we make some complementary remarks on different Poincar\'e inequalities. While not essential for our main results, these arise as natural by-products, and may have some independent interest.

\section{Preliminaries}\label{sec:prelim}

We will work over a metric measure space $(X,\rho,\mu)$. As usual, we write $A\lesssim B$, if $A\leq c\cdot B$ for some constant $c$ that only depends on the space $(X,\rho,\mu)$ and some related fixed parameters, but never on individual points $x\in X$, functions $f$, or similar variable quantities. We write $A\sim B$ if both $A\lesssim B$ and $B\lesssim A$.

We always assume the doubling condition
\begin{equation*}
   V(x,2r)\leq c_\mu V(x,r),\qquad V(x,r):=\mu(B(x,r)),
\end{equation*}
for all $x\in X$ and $r>0$. For $x,y\in X$, we also denote
\begin{equation*}
  V(x,y):=V(x,\rho(x,y))\sim V(y,x),
\end{equation*}
where the last equivalence is immediate from the doubling condition.

\begin{definition}\label{def:dims}
A metric measure space $(X,\rho,\mu)$ is said to have {\em lower dimension $d$} and {\em upper dimension $D$} if
\begin{equation*}
   \Big(\frac{R}{r}\Big)^d\lesssim \frac{V(x,R)}{V(x,r)}\lesssim\Big(\frac{R}{r}\Big)^D
\end{equation*}
for all $x\in X$ and $0<r\leq R<2\operatorname{diam}(X)$. The space is said to be {\em Ahlfors $d$-regular} if it has lower and upper dimension $d=D$.
\end{definition}

Every doubling space has upper dimension $D\leq\log_2 c_\mu$, but possibly much lower than this simple upper bound.
While not a property of every doubling space, a lower dimension $d>0$ readily follows in the presence of a Poincar\'e inequality (Definition \ref{def:Poincare}), which we assume in our main results: a space supporting a Poincar\'e inequality is connected by \cite[Proposition 4.2]{BB:book}, and a connected doubling space has lower dimension $d>0$ by \cite[Corollary 3.8]{BB:book}.

\begin{remark}\label{rem:Ahlfors}
The above definition of Ahlfors $d$-regular is equivalent to the usual one: a space if Ahlfors $d$-regular if and only if $V(x,r)\sim r^d$ for all $x\in X$ and $0<r<2\operatorname{diam}(X)$.
\end{remark}

\begin{proof}
``If'' is clear. For ``only if'', suppose that $V(x,R)/V(x,r)\sim (R/r)^d$ for all $x\in X$ and $0<r\leq R<2\operatorname{diam}(X)$. It is immediate that this implies doubling, and hence $V(x,R)\sim V(y,R)$ for $R\geq\rho(x,y)$. Given $x,y\in X$ and $0<r<2\operatorname{diam}(X)$, we can choose $R\in[r,2\operatorname{diam}(X))$ such that $R\sim r+\rho(x,y)$ to obtain
\begin{equation*}
  V(x,r)\sim\Big(\frac{r}{R}\Big)^d V(x,R)
  \sim\Big(\frac{r}{R}\Big)^d V(y,R)\sim V(y,r).
\end{equation*}
Fixing some $x_0\in X$, it follows that
\begin{equation*}
  V(x,r)\sim V(x_0,r)\sim\Big(\frac{r}{1}\Big)^d V(x_0,1)\sim r^d.
\end{equation*}
\end{proof}

When speaking of {\em local} properties of functions on a metric measure space $(X,\rho,\mu)$, we mean properties of their restrictions on any ball $B=B(x,r)$ with $x\in X$ and $r\in(0,\infty)$. In particular, $L^1_{\loc}(\mu)$ consists of all measurable $f$ such that $1_B f\in L^1(\mu)$ for every ball $B$, and $\Lip_{\loc}$ consists of all functions $f$ whose restriction $f|_B$ is a Lipschitz function (as defined right below) for every ball~$B$.

\subsection{Lipschitz functions and the Poincar\'e inequality}

For a function $f$ on $X$, its Lipschitz (semi-)norm is the quantity
\begin{equation*}
  \Norm{f}{\Lip}:=\sup_{\substack{x,y\in X\\ x\neq y}} \frac{\abs{f(x)-f(y)}}{\rho(x,y)},
\end{equation*}
and $f$ is said to be Lipschitz continuous, or a Lipschitz function, if $\Norm{f}{\Lip}<\infty$.

The pointwise Lipschitz constant of $f$ at $x\in X$ is defined by
\begin{equation}\label{eq:lipf}
  \operatorname{lip}f(x):=\liminf_{t\to 0}\sup_{\rho(x,y)\leq t}\frac{\abs{f(x)-f(y)}}{t}.
\end{equation}

A key assumption in our main results is the following:

\begin{definition}\label{def:Poincare}
For $p\in[1,\infty)$, we say that $(X,\rho,\mu)$ supports a $(1,p)$-Poincar\'e inequality if, for some constants $\lambda\in[1,\infty)$ and $c_P\in(0,\infty)$, we have
\begin{equation}\label{eq:Poincare}
   \fint_{B(z,r)}\abs{f-\ave{f}_{B(z,r)}}\ud\mu
   \leq c_P\cdot r\cdot\Big(\fint_{B(z,\lambda r)}(\operatorname{lip}f)^p\ud\mu\Big)^{\frac1p},
\end{equation}
for all $z\in X$ and $r>0$, and all Lipschitz functions $f$ on $X$.
\end{definition}

For the sake of comparison, we relate the notion of Poincar\'e inequality given in Definition \ref{def:Poincare} to some other variants appearing in the literature.

\begin{lemma}\label{lem:Poincares}
Let $(X,\rho,\mu)$ be a complete doubling metric measure space and $p\in[1,\infty)$. Consider the Poincar\'e inequality
\begin{equation}\label{eq:uPoincare}
  \fint_{B(x,r)}\abs{f-\ave{f}_{B(x,r)}}\ud\mu\lesssim r\Big(\fint_{B(x,\lambda r)}g^p\ud\mu\Big)^{\frac1p}
\end{equation}
for pairs of functions $(f,g)$. Then the following statements---each to hold for every Lipschitz function $f$ on $X$, and $g$ as specified below---are equivalent:
\begin{enumerate}[\rm(1)]
  \item\label{it:lipf} \eqref{eq:uPoincare} holds for $g(x)=\operatorname{lip}f(x)$ as defined in \eqref{eq:lipf}.
  \item\label{it:Lipf} \eqref{eq:uPoincare} holds for 
$\displaystyle  g(x)=\operatorname{Lip}f(x):=\limsup_{r\to 0}\sup_{0<\rho(x,y)\leq r}\frac{\abs{f(x)-f(y)}}{\rho(x,y)}.$
  \item\label{it:itLipf} \eqref{eq:uPoincare} holds for 
$\displaystyle  g(x)=\operatorname{\it Lip}f(x):=\liminf_{r\to 0}\sup_{\rho(x,y)=r}\frac{\abs{f(x)-f(y)}}{r}.$
  \item\label{it:uppergrad} \eqref{eq:uPoincare} holds whenever $g$ is an upper gradient of $f$.
\end{enumerate}
\end{lemma}

\begin{proof}
Since $\operatorname{\it Lip}f\leq\operatorname{lip}f\leq\operatorname{Lip}f$, the implications \eqref{it:itLipf}$\Rightarrow$\eqref{it:lipf}$\Rightarrow$\eqref{it:Lipf} are obvious. The implication \eqref{it:Lipf}$\Rightarrow$\eqref{it:uppergrad} is part of \cite[Theorem 2]{Keith:03}.
By \cite[Proposition 1.11]{Cheeger}, the function $\operatorname{\it Lip}f$ is an upper gradient of a Lipschitz function $f$. Hence \eqref{it:uppergrad}$\Rightarrow$\eqref{it:itLipf}, and this closes the circle of implications.
\end{proof}

We also make use of the following fundamental result.

\begin{theorem}[\cite{KZ:08}, Theorem 1.0.1]\label{thm:KZ}
Let $(X,\rho,\mu)$ be a complete doubling metric measure space and $p\in(1,\infty)$.
Suppose that $(X,\rho,\mu)$ satisfies the $(1,p)$-Poincar\'e inequality in the sense of Definition \ref{def:Poincare}. Then it satisfies the $(1,p-\eps)$-Poincar\'e inequality for some $\eps>0$, in the same sense.
\end{theorem}

\begin{proof}
As indicated, this was proved in \cite[Theorem 1.0.1]{KZ:08}, but under a definition of the Poincar\'e inequality as in Lemma \ref{lem:Poincares}\eqref{it:Lipf} (see \cite[Definition 2.2.1]{KZ:08}). By Lemma \ref{lem:Poincares}, this is equivalent to any of the other definitions listed in the lemma, and in particular to the version of Lemma \ref{lem:Poincares}\eqref{it:lipf} that was used in our Definition \ref{def:Poincare}.
\end{proof}

Various examples of spaces supporting a Poincar\'e inequality can be found in \cite{BBG}, \cite[Section 14.2]{HKST}, and \cite[page 274]{Keith:04}. 

\begin{remark}\label{rem:KZ}
It is only through the application of Theorem \ref{thm:KZ} that the assumption of completeness enters into our considerations. An interested reader can verify that each occurrence of the assumption of ``a complete space with a $(1,p)$-Poincar\'e inequality'' could be replaced by assuming ``a $(1,p-\eps)$-Poincar\'e inequality for some $\eps>0$'' throughout the paper. Note that this only concerns Theorem \ref{thm:Rupert} on the characterisation of the Sobolev space, while Theorem \ref{thm:const-intro} on the characterisation of constants is already independent of completeness as stated.
\end{remark}

\subsection{Haj\l{}asz gradients and the space $\dot M^{1,p}(\mu)$}

While several meaningful notions of a first-order Sobolev space over a metric measure space are known by now, the most straightforward definition is probably that of Haj\l{}asz \cite{Hajlasz:96}:

\begin{definition}\label{def:Hajlasz}
A measurable function $h$ is a {\em Haj\l{}asz upper gradient} of a measurable function $f$ if
\begin{equation}\label{eq:HajlaszUpper}
   \abs{f(x)-f(y)}\leq(h(x)+h(y))\rho(x,y),\qquad\text{for $\mu$-a.e. }x,y\in X.
\end{equation}
The homogeneous Haj\l{}asz--Sobolev norm of $f$ is then
\begin{equation*}
   \Norm{f}{\dot M^{1,p}(\mu)}
  :=\inf\Big\{\Norm{h}{L^p(\mu)}: h\text{ satisfies \eqref{eq:HajlaszUpper}}\Big\},
\end{equation*}
and $\dot M^{1,p}(\mu)$ is the class of all measurable $f$ with $\Norm{f}{\dot M^{1,p}(\mu)}<\infty$.
\end{definition}

\begin{example}\label{ex:HajlaszRd}
By \cite[Theorem 1]{Hajlasz:96}, we have
\begin{equation*}
    \dot M^{1,p}(\R^d)=\dot W^{1,p}(\R^d),\qquad p\in(1,\infty),
\end{equation*}
where $\dot W^{1,p}(\R^d)$ is the usual homogeneous Sobolev space
\begin{equation*}
  \dot W^{1,p}(\R^d):=\Big\{f\in L^1_{\loc}(\R^d):\Norm{f}{\dot W^{1,p}(\R^d)}:=\Norm{\nabla f}{L^p(\R^d)}<\infty\Big\},
\end{equation*}
and $\nabla f$ is the distributional gradient. This shows that \cite[Theorem~1]{Frank} falls under the scope of Theorem \ref{thm:Rupert}---except for details of the limit relation \eqref{eq:Rupert}, which takes a sharper form in $\R^d$.
\end{example}

\begin{remark}\label{rem:SobolevSpaces}
Under the assumptions of Theorem \ref{thm:Rupert}, the Haj\l{}asz--Sobolev space $M^{1,p}(\mu)$ is equivalent so several other established notions of Sobolev spaces over a metric measure space, including the Cheeger Sobolev space $\textit{Ch}^{1,p}$ of \cite{Cheeger}, the Korevaar--Schoen Sobolev space $\textit{KS}^{1,p}$ of \cite{KS:93}, the Newtonian Sobolev space $N^{1,p}$ of \cite{Nages:00}, and the Poincar\'e Sobolev space $P^{1,p}$ of \cite{HK:00}, as well as various Sobolev spaces $H^{1,p}$ associated with given systems of vector fields, which is often the preferred definition in concrete situations. See  \cite[Theorem 10.5.3]{HKST} for $M^{1,p}=P^{1,p}=\textit{KS}^{1,p}=N^{1,p}=\textit{Ch}^{1,p}$, and \cite[Theorem 10]{FHK:99} for $H^{1,p}=P^{1,p}$ under the assumptions of Theorem \ref{thm:Rupert} on $(X,\rho,\mu)$, and quite general assumptions on the vector fields defining $H^{1,p}$ in the last equality. The said results are formulated for the inhomogeneous versions of these spaces (involving the size of both the function and its ``gradient''), but they can be readily adapted to the homogeneous versions relevant to Theorem \ref{thm:Rupert}. (Note also that, by Theorem \ref{thm:KZ}, the assumptions of Theorem \ref{thm:Rupert} imply those of  \cite[Theorem 10.5.3]{HKST}.)
\end{remark}

\begin{example}\label{ex:Carnot}
Any Heisenberg group $\mathbb H^n$, and more generally any Carnot group $\mathbb G$ equipped with the Carnot--Carath\'eodory metric $d_{cc}$ and Haar measure $\mu$, is a doubling metric measure space supporting the $(1,1)$-Poincar\'e inequality. Moreover, the Newtonian Sobolev space $N^{1,p}(\mathbb G)$ (and thus the Haj\l{}asz Sobolev space $M^{1,p}(\mathbb G)$ by Remark \ref{rem:SobolevSpaces}) is identified with the {\em horizontal Sobolev space} $W^{1,p}_H(\mathbb G)$; see \cite[Section 14.2]{HKST} for details and additional pointers to the literature. This shows that \cite[Theorem A.3]{LXY} falls under the scope of Theorem \ref{thm:Rupert}.
\end{example}

\section{The macroscopic Poincar\'e inequality}\label{sec:macroPoincare}

In this section, we establish our key new tool, the ``macroscopic'' Poincar\'e inequality of Theorem \ref{thm:macroPoincare}, which plays a role in the characterisation of both Sobolev spaces and constant functions, and is likely to find other applications elsewhere. We first recall some well known facts about the construction of approximate identities on metric spaces.

\begin{lemma}\label{lem:partUnity}
Let $(X,\rho,\mu)$ be a doubling metric measure space, and fix a parameter $\gamma>0$.
Then, for every $t>0$, there exist a countable family of functions $\phi_i$ and points $x_i$ such that the balls $B(x_i,\frac12 t)$ are disjoint, and
\begin{equation*}
  0\leq \phi_i\leq 1_{B(x_i,(1+\gamma)t)},\qquad
  \sum_i\phi_i\equiv 1,\qquad\Norm{\phi_i}{\Lip}\lesssim t^{-1}
\end{equation*}
\end{lemma}

\begin{proof}
We choose a $t$-net of points $x_i\in X$. Since $(X,\rho,\mu)$ is doubling, this net is locally finite. For each $i$, let $\tilde\phi_i$ be a function with $1_{B(x_i,t)}\leq\tilde\phi_i\leq 1_{B(x_i,(1+\gamma)t)}$ and $\Norm{\tilde\phi_i}{\Lip}\leq (\gamma t)^{-1}$. Then $1\leq \sum_i\tilde\phi_i\leq c$, where $c$ depends only on the doubling constant and $\gamma$. Then $\phi_i:=\tilde\phi_i/\sum_j\tilde\phi_j$ is well-defined, and satisfies the required properties.
\end{proof}

For every $t>0$, we define an approximate identity $\Phi_t:L^1_{\loc}(\mu)\to\Lip_{\loc}(X)$ with the help of a fixed partition of unity $\phi_i$ at scale $t$ as provided by Lemma \ref{lem:partUnity}:
\begin{equation}\label{eq:Phi-tf}
  \Phi_t f:=\sum_i \phi_i(x)\ave{f}_{B(x_i,t)}.
\end{equation}
The local Lipschitz properties of $\Phi_t f$ are quantified in the following:

\begin{lemma}\label{lem:PhitDiff}
Let $(X,\rho,\mu)$ be a doubling metric measure space. For any fixed parameter $a>0$, there exists another parameter $b>0$ such that the following estimates hold for every $f\in L^1_{\loc}(\mu)$ and $t>0$:
If $\rho(x,y)\leq a t$, the approximations $\Phi_t f$ from \eqref{eq:Phi-tf} satisfy the estimate
\begin{equation*}
  \abs{\Phi_t f(x)-\Phi_t f(y)}\lesssim\frac{\rho(x,y)}{t} m_f(x,bt)
\end{equation*}
and in particular
\begin{equation*}
  \Lip \Phi_t f(x)\lesssim \frac{1}{t}m_f(x,bt).
\end{equation*}
\end{lemma}

\begin{proof}
Denoting $c_i:=\ave{f}_{B(x_i,t)}$, we first observe the identity
\begin{equation*}
\begin{split}
  \Phi_t f(x) & -\Phi_t f(y) \\
  &=\sum_i\phi_i(x)c_i-\sum_j\phi_j(y)c_j \\
  &=\sum_{i,j}\phi_i(x)\phi_j(y)(c_i-c_j)\qquad\text{since }\sum_i\phi_i\equiv 1 \\
  &=\sum_{i,j}\phi_j(x)\phi_i(y)(c_j-c_i)\qquad\text{swapping the names of the indices} \\
  &=\frac12\sum_{i,j}\Big(\phi_i(x)\phi_j(y)-\phi_j(x)\phi_i(y)\Big)(c_i-c_j)
\end{split}
\end{equation*}
by taking the average of the previous two equal expressions in the last step.

In order that
\begin{equation*}
  \phi_i(x)\phi_j(y)-\phi_j(x)\phi_i(y)\neq 0
\end{equation*}
it is necessary that $\phi_i(x)\neq 0$ or $\phi_i(y)\neq 0$, thus $x_i\in B(x,(1+\gamma)t)$ or $x_i\in B(y,(1+\gamma)t)\subset B(x,(1+\gamma+a) t)$; hence in any case $x_i\in B(x,(1+\gamma+a) t)$. Similarly, we deduce that $x_j\in B(x,(1+\gamma+a) t)$. In particular, the relevant summations over both $i$ and $j$ have only boundedly many nonzero terms.

Moreover, it follows $B(x_i,t)\subset B(x,(1+\gamma+2a) t)\subset B(x_i,2(1+\gamma+a) t)$, and hence
\begin{equation*}
\begin{split}
  &\abs{c_i-c_j}
  =\abs{\ave{f}_{B(x_i,t)}-\ave{f}_{B(x_j,t)}} \\
  &\quad=\Babs{\fint_{B(x_i,t)}(f-\ave{f}_{B(x,(1+\gamma+2a) t)})\ud\mu-\fint_{B(x_j,t)}(f-\ave{f}_{B(x,(1+\gamma+2a) t)})\ud\mu} \\
  &\quad\lesssim \fint_{B(x,b t)}\abs{f-\ave{f}_{B(x,b t)}}\ud\mu,\qquad b:=1+\gamma+2a.
\end{split}
\end{equation*}
On the other hand, we have
\begin{equation*}
\begin{split}
  \abs{\phi_i(x)\phi_j(y)-\phi_j(x)\phi_i(y)}
  &=\abs{(\phi_i(x)-\phi_i(y))\phi_j(y)+\phi_i(y)(\phi_j(y)-\phi_j(x))} \\
  &\leq \rho(x,y)\Norm{\phi_i}{\Lip}\Norm{\phi_j}{\infty}+\Norm{\phi_i}{\infty}\rho(x,y)\Norm{\phi_j}{\Lip} \\
  &\lesssim\frac{\rho(x,y)}{t}.
\end{split}
\end{equation*}
Altogether, we have
\begin{equation*}
  \abs{\Phi_t f(x)-\Phi_t f(y)}\lesssim \frac{\rho(x,y)}{t}\fint_{B(x,b t)}\abs{f-\ave{f}_{B(x,b t)}}\ud\mu
\end{equation*}
as claimed. The claim concerning $\lip\Phi_t f(x)$ is an immediate consequence.
\end{proof}

We are now ready to prove the macroscopic Poincar\'e inequality:


\begin{theorem}\label{thm:macroPoincare}
Let $p\in[1,\infty)$ and let $(X,\rho,\mu)$ be doubling metric measure space supporting a $(1,p)$-Poincar\'e inequality with constants $(\lambda,c_P)$. Then for all $f\in L^1_{\loc}(\mu)$, all $z\in X$, and all $0<r<s<\infty$, we have 
\begin{equation*}
  m_f(z,s)\lesssim\frac{s}{r}\Big(\fint_{B(z,\tilde\lambda s)} m_f(x,r)^p\ud\mu(x)\Big)^{\frac1p},
\end{equation*}
where $\tilde\lambda:=\max(2,\lambda)$ and the implied constant depends only on the doubling and Poincar\'e data.
\end{theorem}

\begin{proof}
It is enough to prove the claim with some number $c$ in place of $\ave{f}_{B(z,s)}$ in the expression
\begin{equation*}
  m_f(z,s)=\fint_{B(z,s)}\abs{f-\ave{f}_{B(z,s)}}\ud\mu.
\end{equation*}
For $t\sim r$ to be specified, we consider an approximate identity $\Phi_t$ at scale $t$ and estimate
\begin{equation*}
  \fint_{B(z,s)}\abs{f-c}\ud\mu
  \leq\fint_{B(z,s)}\abs{f-\Phi_t f}\ud\mu+\fint_{B(z,s)}\abs{\Phi_t f-c}\ud\mu=:I+II.
\end{equation*}

We deal with term $II$ first. Taking $c=\ave{\Phi_t f}_{B(z,s)}$, we have
\begin{equation}\label{eq:macroII}
\begin{split}
  II &\lesssim s\Big(\fint_{B(z,\lambda s)}[\Lip\Phi_t f(x)]^p\ud\mu(x)\Big)^{\frac1p}\qquad\text{(Poincar\'e)} \\
  &\lesssim s\Big(\fint_{B(z,\lambda s)}\Big[\frac{1}{t}m_f(x,bt)\Big]^p\ud\mu(x)\Big)^{\frac1p}
  \quad\text{(Lemma \ref{lem:PhitDiff})}.
\end{split}
\end{equation}

We then turn to term $I$. Denoting $c_i:=\ave{f}_{B(x_i,t)}$, we have
\begin{equation*}
  f-\Phi_t f=\sum_i\phi_i(f-c_i),
\end{equation*}
where $\phi_i$ is a partition of unity at scale $t$ as in Lemma \ref{lem:partUnity}.
Hence
\begin{equation*}
  I\lesssim \frac{1}{\mu(B(z,s))}\sum_{i:B(x_i,(1+\gamma) t)\cap B(z,s)\neq\varnothing}\mu(B(x_i,(1+\gamma)t))\fint_{B(x_i,(1+\gamma) t)}\abs{f-c_i}\ud\mu.
\end{equation*}

For each $x\in B(x_i,\frac12 t)$, we have $$B(x_i,(1+\gamma)t)\subset B(x,(\frac32+\gamma) t)\subset B(x_i,(2+\gamma)t),$$ and hence
\begin{equation*}
   \fint_{B(x_i,(1+\gamma)t)}\abs{f-\ave{f}_{B(x_i,t)}}\ud\mu
   \lesssim \fint_{B(x,\beta t)}\abs{f-\ave{f}_{B(x,\beta t)}}\ud\mu,\quad \beta:=\frac32+\gamma.
\end{equation*}
Thus (abbreviating the summation condition $B(x_i,(1+\gamma)t)\cap B(z,s)\neq\varnothing$ by ``$*$'')
\begin{equation*}
\begin{split}
  I &\lesssim\sum_i^*\frac{\mu(B(x_i,\frac12 t))}{\mu(B(z,s))}\inf_{x\in B(x_i,\frac12 t)}m_f(x,\beta t) \\
  &\lesssim\frac{1}{\mu(B(z,s))}\sum_i^*\int_{B(x_i,\frac12 t)}m_f(x,\beta t)\ud\mu(x)
\end{split}
\end{equation*}
The balls $B(x_i,\frac12 t)$ are disjoint and, under the assumption that $B(x_i,(1+\gamma)t)\cap B(z,s)\neq\varnothing$, contained in $B(z,s+(\frac32+\gamma) t)=B(z,s+\beta t)$. Hence
\begin{equation*}
  I \lesssim \frac{1}{\mu(B(z,s))}\int_{B(z,s+\beta t)}m_f(x,\beta t)  \ud\mu(x).
\end{equation*}

Let us now fix choose $c:=\max(b,\beta)$ and $t:=c^{-1}r$. Then $s+\beta t\leq s+r\leq 2s$, and we have proved that
\begin{equation*}
\begin{split}
  II &\lesssim\frac{s}{r}\Big(\fint_{B(z,\lambda s)} m_f(x,r)^p
  \ud\mu(x)\Big)^{\frac1p}, \\
  I &\lesssim \fint_{B(z,2s)}m_f(x,r)\ud\mu(x) 
   \leq \frac{s}{r}\Big(\fint_{B(z,2s)}m_f(x,r)^p\ud\mu(x)\Big)^{\frac1p}
\end{split}
\end{equation*}
by H\"older's inequality and the trivial bound $1\leq s/r$ in the last step.

By yet another application of doubling, we can dominate the averages over both $B(z,\lambda s)$ and $B(z,2s)$ by one over $B(z,\max(2,\lambda) s)$. Thus both terms satisfy the claimed bound.
\end{proof}

\begin{remark}\label{rem:macroPoincare}
In the proof of Theorem \ref{thm:macroPoincare}, the assumed $(1,p)$-Poincar\'e inequality was only used in \eqref{eq:macroII} in a weaker form involving the larger $\Lip$ on the right-hand side, instead of the smaller $\lip$ as in Definition \eqref{def:Poincare}. By Lemma \ref{lem:Poincares}, these versions are equivalent if $X$ is complete, but completeness is not assumed in Theorem \ref{thm:macroPoincare}.
\end{remark}

\section{How to recognise constants}\label{sec:constants}

With the macroscopic Poincar\'e inequality in our toolbox, we are in a position to give a relatively quick proof of the characterisation of constants via Besov-type conditions as in Theorem \ref{thm:const-intro}. We have the following slightly stronger version of case \eqref{it:B1dd=const} of the said theorem:

\begin{proposition}\label{prop:constBesov}
Let $p\in[1,\infty)$ and let $(X,\rho,\mu)$ be a doubling metric measure space supporting the $(1,p)$-Poincar\'e inequality.
If $f\in L^1_{\loc}(\mu)$ satisfies
\begin{equation}\label{eq:criticalBesov}
  \Big[(x,y)\mapsto  \frac{\abs{f(x)-f(y)}^p}{\rho (x,y)^p}\frac{1}{V(x,y)}\Big]\in L^1_{\loc}(\mu\times\mu),
\end{equation}
then $f$ is equal to a constant almost everywhere. In particular, the Besov space defined in Theorem \ref{thm:const-intro} satisfies
\begin{equation*}
  \dot B^1_{p,p}(\mu)=\{\operatorname{constants}\}.
\end{equation*}
\end{proposition}

\begin{proof}
Note that $f\in \dot B^1_{p,p}(\mu)$ means, by definition, that the function in \eqref{eq:criticalBesov} is not only $L^1_{\loc}(\mu\times\mu)$ but globally in $L^1(\mu\times\mu)$; hence it is enough to prove the first assertion.

It suffices to show that, for every ball $B(z,s)\subset X$, we have $f\equiv\ave{f}_{B(z,s)}$ almost everywhere in $B(z,s)$. We fix one such ball and apply Theorem \ref{thm:macroPoincare} to see that, for every $r\in(0,s)$, and using the notation of the said proposition,
\begin{align*}
  &\fint_{B(z,s)}\abs{f-\ave{f}_{B(z,s)}}\ud\mu \\
  &\lesssim\frac{s}{r}\Big(\fint_{B(z,\tilde\lambda s)}\Big[\fint_{B(x,r)}\abs{f(y)-\ave{f}_{B(x,r)}}\ud\mu(y)\Big]^p\ud\mu(x)\Big)^{\frac1p} \\
  &\leq\frac{s}{r}\Big(\fint_{B(z,\tilde\lambda s)}\fint_{B(x,r)}\fint_{B(x,r)}\abs{f(y)-f(v)}^p\ud\mu(y)\ud\mu(v)\ud\mu(x)\Big)^{\frac1p} \\
  &=\frac{s}{r}\Big(\int_{x\in B(z,\tilde\lambda s)}\int_{v\in B(x,r)}\int_{y\in B(x,r)}\frac{\abs{f(y)-f(v)}^p}{V(x,r)^2}
    \frac{\ud\mu(y)\ud\mu(v)\ud\mu(x)}{V(z,\tilde\lambda s)}\Big)^{\frac1p} \\
  &\lesssim\frac{s}{r}\Big(\int_{y\in B(z,\tilde\lambda s+r)}\int_{v\in B(y,2r)}\int_{x\in B(y,r)}\frac{\abs{f(y)-f(v)}^p}{V(y,r)^2}\ud\mu(x)
  \frac{\ud\mu(v)\ud\mu(y)}{V(z,\tilde\lambda s)}\Big)^{\frac1p} \\
  &=\frac{s}{r}\Big(\int_{y\in B(z,(\tilde\lambda+1)s)}\int_{v\in B(y,2r)}
  \frac{\abs{f(y)-f(v)}^p}{V(y,r)}\frac{\ud\mu(v)\ud\mu(y)}{V(z,\tilde\lambda s)}\Big)^{\frac1p} \\
  &\lesssim\frac{s}{r}\Big(\int_{y\in B(z,(\tilde\lambda+1)s)}\int_{v\in B(y,2r)}\Big(\frac{2r}{\rho(v,y)}\Big)^p
  \frac{\abs{f(y)-f(v)}^p}{V(v,y)}\frac{\ud\mu(v)\ud\mu(y)}{V(z,\tilde\lambda s)}\Big)^{\frac1p} \\
  &\lesssim s\Big(\iint_{(v,y)\in B(z,(\tilde\lambda+3)s)^2}1_{\{\rho(v,y)< 2r\}}
  \frac{\abs{f(y)-f(v)}^p}{\rho(v,y)^p V(v,y)}\frac{\ud\mu(v)\ud\mu(y)}{V(z,\tilde\lambda s)}\Big)^{\frac1p}.
\end{align*}
In the rightmost form, the integrand is pointwise dominated by the function 
\begin{equation*}
  (v,y)\mapsto\frac{\abs{f(y)-f(v)}^p}{\rho(v,y)^p}\frac{1}{V(v,y)},
\end{equation*}
which is integrable over the bounded set $B(z,(\tilde\lambda+3)s)^2\subset X\times X$ by assumption \eqref{eq:criticalBesov}. As $r\to 0$, it is evident that we have $1_{\{\rho(v,y)<2r\}}\to 0$ pointwise (except when $v=y$, but in this case the integrand vanishes anyway due to the factor $\abs{f(y)-f(v)}^p$), and hence the integral above converges to $0$ by dominated convergence. This shows that
\begin{equation*}
   \fint_{B(z,s)}\abs{f-\ave{f}_{B(z,s)}}\ud\mu = 0,
\end{equation*}
and hence $f\equiv \ave{f}_{B(z,s)}$ almost everywhere in $B(x,s)$, as we wanted to show.
\end{proof}

Under the additional assumption of {\em lower dimension} $d$ as in Definition \ref{def:dims} we also obtain the following variant of the constancy criterion, which turns out to be natural in view of the commutator applications in \cite{Hyt:Sp}.

\begin{corollary}\label{cor:constBesov}
Let $1\leq d<\infty$ and let $(X,\rho,\mu)$ be a doubling metric measure space having lower dimension $d$ and supporting the $(1,d)$-Poincar\'e inequality.
If $0<p\leq d$ and $f\in L^p_{\loc}(\mu)$ satisfies
\begin{equation}\label{eq:criticalBesov2}
  \Big[(x,y)\mapsto  \frac{\abs{f(x)-f(y)}^p}{V(x,y)^2}\Big]\in L^1_{\loc}(\mu\times\mu),
\end{equation}
then $f$ is equal to a constant almost everywhere. In particular, the Besov space defined in Theorem \ref{thm:const-intro} satisfies
\begin{equation*}
   \dot B^p(\mu)=\{\operatorname{constants}\}\qquad\forall\ p\in(0,d].
\end{equation*}
\end{corollary}

\begin{proof}
\textbf{Case $p=d$.} Fix a ball $B_0:=B(x_0,R)$. For all $x,y\in B_0$, we have
\begin{equation*}
  \frac{\mu(B_0)}{V(x,y)}\sim\frac{V(x,R)}{V(x,y)}
  \gtrsim\Big(\frac{R}{\rho(x,y)}\Big)^d
\end{equation*}
by doubling in the first step and lower dimension $d$ in the second one.
Hence
\begin{equation*}
  \frac{1}{\rho(x,y)^d}\lesssim\frac{1}{V(x,y)}\frac{\mu(B_0)}{R^d},
\end{equation*}
and thus
\begin{equation*}
  \frac{\abs{f(x)-f(y)}^d}{\rho(x,y)^d V(x,y)}
  \lesssim\frac{\abs{f(x)-f(y)}^d}{V(x,y)^2}\frac{\mu(B_0)}{R^d}.
\end{equation*}
Under the assumption \eqref{eq:criticalBesov2}, the right-hand side is integrable over $B_0\times B_0$, and hence is the left-hand side. Since this is true for every ball $B_0$, we have verified assumption \eqref{eq:criticalBesov} of Proposition \ref{prop:constBesov} with $p=d$, and the said proposition shows that $f$ is constant.

\textbf{Case $p\in(0,d)$.} We reduce this to the previous case with the help of the following elementary observation. For $z\in\C$, or more generally an element of a normed space, we define its truncation at level $K>0$ by
\begin{equation}\label{eq:zK}
  z_K:=\begin{cases} z, & \text{if }\abs{z}\leq K, \\ \displaystyle K\frac{z}{\abs{z}}, & \text{else}.\end{cases}
\end{equation}
Then $\abs{z_K}\leq K$ (obviously), and it is easy to check that $\abs{a_K-b_K}\leq 2\abs{a-b}$ for any two elements $a,b$. Hence
\begin{equation*}
  \abs{a_K-b_K}^d=\abs{a_K-b_K}^{d-p}\abs{a_K-b_K}^p\leq(2K)^{d-p}(2\abs{a-b})^p= 2^d K^{d-p}\abs{a-b}^p.
\end{equation*}

Given $f\in L^p_{\loc}(\mu)$, we consider the truncated functions $f_K$, applying \eqref{eq:zK} pointwise. If $f$ satisfies \eqref{eq:criticalBesov2} for some $p\in(0,d)$, then $f_K\in L^\infty_{\loc}(\mu)\subseteq L^d_{\loc}(\mu)$ satisfies
\begin{equation*}
  \frac{\abs{f_K(x)-f_K(y)}^d}{V(x,y)^2}
  \leq 2^d K^{d-p}\frac{\abs{f(x)-f(y)}^p}{V(x,y)^2}\in L^1_{\loc}(\mu\times\mu).
\end{equation*}
By the case $p=d$ that we already proved, it follows that $f_K$ is equal to some constant $c_K$ almost everywhere. Since $f_K=f=f_{K'}$ on $\{f\leq\max(K,K')\}$, it follows that $c_K=c_{K'}=:c$ is independent of $K$ as soon as $K$ is large enough so that $\{f\leq K\}$ has positive measure. Since $f_K\to f$ pointwise as $K\to\infty$, it follows that $f=c$ almost everywhere.
\end{proof}

We record the following corollary in a specific measure space, improving a recent result of \cite[Proposition 6.4]{FLLX}; see Remark \ref{rem:Bessel} below for comparison and motivation.

\begin{corollary}\label{cor:Bessel}
Let $n\geq 0$ and $\lambda>0$, and consider the measure $\ud m_\lambda(x)=x_{n+1}^{2\lambda}\ud x$ on $\R^{n+1}_+$. Then
\begin{equation*}
   \dot B^p(m_\lambda)=\{\operatorname{constants}\}\qquad\forall\ p\in(0,n+1].
\end{equation*}
\end{corollary}

\begin{proof}
By \cite[Proposition 4.2]{Hyt:Sp}, $(\R^{n+1}_+,\abs{x-y},m_\lambda)$ is a doubling metric measure space of lower dimension $n+1$ and supporting the $(1,1)$ (and hence in particular the $(1,n+1)$) Poincar\'e inequality. Hence Corollary \ref{cor:Bessel} is a specialisation of Corollary \ref{cor:constBesov}, with $d=n+1$, to this concrete setting.
\end{proof}

\begin{remark}\label{rem:Bessel}
The metric measure space $(\R^{n+1},\abs{x-y},m_\lambda)$ of Corollary \ref{cor:Bessel} is the so-called ``Bessel setting'' arising from the analysis of the Bessel differential operator
\begin{equation*}
  \Delta_\lambda:=-\sum_{j=1}^{n+1}\frac{\partial^2}{\partial x_j^2}-\frac{2\lambda}{x_{n+1}}\frac{\partial}{\partial x_{n+1}}.
\end{equation*}
As a tool for related commutator estimates, \cite[Proposition 6.4]{FLLX} recently proved that
\begin{equation}\label{eq:FLLX}
   \dot B^p(m_\lambda)\cap C^2(\R_+^{n+1})=\{\operatorname{constants}\}\qquad\forall\ p\in(0,n+1].
\end{equation}
Corollary \ref{cor:Bessel} shows that the restriction to $C^2$ functions can be removed, illustrating the power of our abstract methods, even when specialised to a concrete case like the Bessel setting.

This is used for commutator estimates as follows. For so-called Bessel--Riesz transforms $R_{\lambda,j}:=\partial_j\Delta_\lambda^{-\frac12}$, \cite[Proposition 5.2]{FLLX} proves that
\begin{equation*}
  \Norm{b}{\dot B^p(m_\lambda)}\lesssim\Norm{[b,R_{\lambda,j}]}{S^p(L^2(m_\lambda))}.
\end{equation*}
In combination with \eqref{eq:FLLX}, this shows that $[b,R_{\lambda,j}]\in S^p(L^2(m_\lambda))$ if and only if $b$ is constant---under the {\em a priori} assumption that $b\in C^2(\R^{n+1}_+)$. With Corollary \ref{cor:Bessel} in place of \eqref{eq:FLLX}, this {\em a priori} assumption may be dropped. Using the abstract Corollary \ref{cor:constBesov} instead, this reasoning is extended to more general domains $(X,\rho,\mu)$ and operators $T$ in place of $(\R_+^{n+1},\abs{x-y},m_\lambda)$ and $R_{\lambda,j}$ in \cite{Hyt:Sp}.

For these commutator applications, one only needs case $p=d>1$ of Corollary \ref{cor:constBesov} (which is all that we had in an earlier version of this paper). Namely, if $[b,T]\in S^p(L^2(\mu))$ for some $p\in(0,d]$ and $d>1$, then we can use the elementary inclusion $S^p\subseteq S^q$ for $p\leq q$ of the Schatten classes to conclude that $[b,T]\in S^d(L^2(\mu))$ as well. Then it suffices to know that $\Norm{b}{\dot B^d(\mu)}\lesssim\Norm{[b,T]}{S^d(\mu)}$ (as is proved under quite general assumptions in \cite{Hyt:Sp}; this needs, however, $d>1$ with strict inequality) and that $\dot B^d(\mu)=\{\operatorname{constants}\}$ (guaranteed by case $p=d$ of Corollary \ref{cor:constBesov}; this is valid for $d\geq 1$, with equality allowed) to conclude that $b=\operatorname{constant}$.
\end{remark}


\section{The lower bound for the Sobolev norm}\label{sec:MpLowerBd}

Our aim in the rest of the paper is to prove the characterisation of the Haj\l{}asz--Sobolev space formulated Theorem \ref{thm:Rupert}, our main result. We begin with the easy direction. This short section is a relatively straightforward adaptation of \cite[Section 2]{Frank}. Moreover, this direction is valid under much more general assumptions than the full Theorem \ref{thm:Rupert}.

\begin{proposition}\label{prop:MpLowerBd}
Let $(X,\rho,\mu)$ be a space of homogeneous type, i.e., a quasi-metric space with a doubling measure.
If $f\in\dot M^{1,p}(\mu)$, then $m_f\in L^{p,\infty}(\nu_p)$ and
\begin{equation*}
  \Norm{m_f}{L^{p,\infty}(\nu_p)}\lesssim \Norm{f}{\dot M^{1,p}(\mu)}.
\end{equation*}
\end{proposition}

\begin{proof}
Let $h$ be a Haj\l{}asz upper gradient of $f$.
For $y,z\in B(x,t)$, it follows from definition \eqref{eq:HajlaszUpper} of the Haj\l{}asz upper gradient and the quasi-triangle inequality
 $\rho(y,z)\leq A_0(\rho(y,x)+\rho(x,z))$ that
\begin{equation*}
  \abs{f(y)-f(z)}
  \leq \rho(y,z)(h(y)+h(z))
  \leq 2A_0 t(h(y)+h(z)),
\end{equation*}
and hence
\begin{equation*}
\begin{split}
  m_f(x,t)
  &\leq\fint_{B(x,t)}\fint_{B(x,t)}\abs{f(y)-f(z)}\ud\mu(y)\ud\mu(z) \\
  &\leq 2A_0t \fint_{B(x,t)}\fint_{B(x,t)} (h(y)+h(z))\ud\mu(y)\ud\mu(z) \\
  &= 2A_0t \Big(\fint_{B(x,t)}h(y)\ud\mu(y)+\fint_{B(x,t)} h(z)\ud\mu(z)\Big) \\
  &= 4A_0t\fint_{B(x,t)}h\ud\mu
  \leq 4A_0t Mh(x),
\end{split}
\end{equation*}
where $M$ is the Hardy--Littlewood maximal operator.

Thus
\begin{equation*}
 \{(x,t):m_f(x,t)>\kappa\}\subseteq\{(x,t):  4A_0t Mh(x)> \kappa\}
\end{equation*}
and
\begin{equation*}
\begin{split}
  p\nu_p(\{m_f>\kappa\})
  &\leq \int_X\int_0^\infty 1_{\{(x,t):  4A_0t Mh(x)> \kappa\}} p t^{-p-1}\ud t\ud\mu(x) \\
  &=\int_X\int_{\kappa/(4A_0 Mh(x))}^\infty p t^{-p-1}\ud t\ud\mu(x) \\
  &=\int_X\Big(\frac{4A_0 Mh(x)}{\kappa}\Big)^p\ud\mu(x) \\
  &\lesssim\frac{1}{\kappa^p}\int_X h(x)^p\ud\mu(x)
\end{split}
\end{equation*}
by the maximal inequality in the last step.

Thus we have proved that
\begin{equation*}
  \sup_{\kappa>0}\kappa^p\nu_p(\{m_f>\kappa\})\lesssim\int_X h^p\ud\mu
\end{equation*}
whenever $h$ is a Haj\l asz upper gradient of $f$. Taking the infimum over all such $h$ gives the claimed bound.
\end{proof}

\section{Maximal versus mean oscillations}\label{sec:Lip-gf}

In contrast to the easy argument of Section \ref{sec:MpLowerBd}, the upper bound of the Sobolev norm is much more delicate and requires some preparations. We note that the Euclidean argument of \cite[Section 3]{Frank} and its Carnot group extension \cite[Appendix~A]{LXY} proceed by first considering test functions of class $C^{1,1}$ (i.e., $f\in C^1$ whose gradient $\nabla f$ is a Lipschitz mapping), and using Taylor expansions with second-order error terms in the estimates. Since these are not available in the present generality, this section develops some alternative tools that have no obvious counterpart in the previous literature.

More precisely, we want to compare the pointwise Lipschitz constant
\begin{equation*}
  \lip f(x) :=\liminf_{r\to 0}\sup_{\rho(x,y)\leq r}\frac{\abs{f(x)-f(y)}}{r}, 
\end{equation*}
describing the {\em maximal oscillations} of $f$ at point $x$ and scale $r\to 0$,
with the {\em a priori} much smaller {\em mean oscillations}, which are more closely connected to the norm $\Norm{m_f}{L^{p,\infty}(\nu_p)}$ in the desired characterisation:
\begin{equation}\label{eq:gf}
   g_f(x):=\liminf_{t\to 0}\frac{1}{t}m_f(x,t),\qquad 
   m_f(x,t):=\fint_{B(x,t)}\abs{f-\ave{f}_{B(x,t)}}\ud\mu.
\end{equation}

\begin{remark}\label{rem:gf-Lip}
For all $f\in L^1_{\loc}(\mu)$ and all $x\in X$, we have
\begin{equation*}
  g_f(x)\leq 2\operatorname{lip}f(x).
\end{equation*}
Indeed,
\begin{equation*}
\begin{split}
    m_f(x,t) &\leq\fint_{B(x,t)}\fint_{B(x,t)}\abs{f(y)-f(z)}\ud\mu(y)\ud\mu(z) \\
    &\leq 2\fint_{B(x,t)}\abs{f(y)-f(x)}\ud\mu(y) 
    \leq 2\sup_{\rho(y,x)\leq t}\abs{f(y)-f(x)},
\end{split}
\end{equation*}
and dividing by $t$ and taking the $\liminf$ of both sides proves the claim.
\end{remark}

A deeper result is the estimate in the converse direction; this is the main result of this section. A similar bound, but with the $\liminf$ in both $\lip f$ and $g_f$ replaced by $\limsup$, is known \cite[Proposition 4.3.3]{Keith:04}; however, a $\liminf$ on the right-hand side is essential for us in view of subsequent estimates in Lemma \ref{lem:gfp<}.

\begin{proposition}\label{prop:Lip-gf}
Let $(X,\rho,\mu)$ be a doubling metric measure space supporting the $(1,p)$-Poincar\'e inequality for some $p\in(1,\infty)$.
If $f$ is a locally Lipschitz function, then 
\begin{equation*}
  \lip f(x)\lesssim g_f(x)
\end{equation*}
$\mu$-almost everywhere, where the implied constant depends only on the doubling constant $c_\mu$ and the parameters in the Poincar\'e inequality.
\end{proposition}

To streamline the proof of the Proposition \ref{prop:Lip-gf}, we start with some lemmas:

\begin{lemma}\label{lem:diff-mf}
Let $(X,\rho,\mu)$ be a doubling metric measure space and $f\in L^1_{\loc}(\mu)$. Then for all $\delta>0$, $K\in\N$, and $u_1,u_2\in X$ with $\rho(u_1,u_2)\leq\delta$, we have
\begin{equation*}
\begin{split}
  &\abs{f(u_1)-f(u_2)} \\
  &\qquad\leq  \sum_{i=1}^2\Big( c_\mu m_f(u_i,3\delta)+c_\mu \sum_{k=0}^K  m_f(u_i,2^{1-k}\delta)+\fint_{B(u_i,2^{-K}\delta)}\abs{f-f(u_i)}\ud\mu\Big).
\end{split}
\end{equation*}
If $u_1$ and $u_2$ are Lebesgue points of $f$, then we also have
\begin{equation*}
  \abs{f(u_1)-f(u_2)}
  \leq c_\mu\sum_{i=1}^2\Big(m_f(u_i,3\delta)+\sum_{k=0}^\infty  m_f(u_i,2^{1-k}\delta)\Big).
\end{equation*}
\end{lemma}

\begin{proof}
We fix $u\in\{u_1,u_2\}$. Then
\begin{equation*}
  f(u_1)-f(u_2)
  =(f(u_1)-\ave{f}_{B(u,3\delta)})-(f(u_2)-\ave{f}_{B(u,3\delta)}).
\end{equation*}
For both $i=1,2$, we have
\begin{equation}\label{eq:basicExp}
\begin{split}
  f(u_i)-\ave{f}_{B(u,3\delta)}
  &=f(u_i)-\ave{f}_{B(u_i,2^{-K}\delta)} \\
  &\qquad+\sum_{k=0}^K\Big(\ave{f}_{B(u_i,2^{-k}\delta)}-\ave{f}_{B(u_i,2^{1-k}\delta)}\Big) \\
  &\qquad+\Big(\ave{f}_{B(u_i,2\delta)}-\ave{f}_{B(u,3\delta)}\Big) \\
  &=:I+\sum_{k=0}^\infty II_k+III.
\end{split}
\end{equation}
It is immediate that
\begin{equation*}
  \abs{I}\leq\fint_{B(u_i,2^{-K}\delta)}\abs{f-f(u_i)}\ud\mu.
\end{equation*}

Let us then consider any of the pairs of balls
\begin{equation*}
  (B',B)\in\{(B(u_i,2^{-k}\delta),B(u_i,2^{1-k}\delta)),(B(u_i,2\delta),B(u,3\delta))\}
\end{equation*}
appearing in the remainder of the expansion. In the first case, $B=2B'$. If $u=u_i$ in the second case, then $B=\frac32 B'$. Thus, in both these cases, we have
\begin{equation}\label{eq:2balls}
  B'\subset B\subset 2B'. 
\end{equation}
We claim that this holds also in the remaining case that $\{u,u_i\}=\{u_1,u_2\}$. In fact, since $B(x,r)\subset B(y,r+\rho(x,y))$ and $\rho(u,u_i)=\delta$, it follows that $B'=B(u_i,2\delta)\subset B(u,3\delta)=B\subset B(u_i,4\delta)=2B'$.

Under the condition \eqref{eq:2balls}, we have
\begin{equation*}
  \abs{\ave{f}_{B'}-\ave{f}_B}  
  \leq\fint_{B'}\abs{f-\ave{f}_B}\ud\mu
  \leq c_\mu\fint_B\abs{f-\ave{f}_B}\ud\mu
  =c_\mu m_f(B).
\end{equation*}
Using this in \eqref{eq:basicExp}, we obtain
\begin{equation*}
  \abs{f(u_i)-\ave{f}_{B(u,3\delta)}}\leq c_\mu\sum_{k=0}^\infty  m_f(u_i,2^{1-k}\delta)+ c_\mu m_f(u, 3\delta),
\end{equation*}
and hence
\begin{equation*}
   \abs{f(u_1)-f(u_2)}\leq c_\mu \sum_{i=1}^2\sum_{k=0}^\infty  m_f(u_i,2^{1-k}\delta)+ 2c_\mu m_f(u,3\delta).
\end{equation*}
This is valid for both choices of $u\in\{u_1,u_2\}$, so taking the minimum of the two bounds gives the first claim of the lemma. The second one follows by taking $K\to\infty$ and using the definition of a Lebesgue point to the integrals over $B(u_i,2^{-K}\delta)$.
\end{proof}

\begin{lemma}\label{lem:diffBound}
Let $f$ be a Lipschitz function on a doubling metric measure space $(X,\rho,\mu)$ that satisfies the $(1,p)$-Poincar\'e inequality with constants $(\lambda,c_P)$. Then for all $u_1,u_2\in X$, all $\delta\geq \rho(u_1,u_2)$ and $\alpha\in(0,1]$, we have
\begin{equation*}
  \abs{f(u_1)-f(u_2)}
  \leq 2\alpha\Norm{f}{\Lip}\delta
  +14C\sup_{\substack{i=1,2 \\ \alpha\lambda\delta\leq s\leq 3\lambda\delta}}
    \Big(\fint_{B(u_i,s)}(\lip f)^p\ud\mu\Big)^{1/p},
\end{equation*}
where $C=c_\mu\cdot c_P$.
\end{lemma}

\begin{proof}
By Lemma \ref{lem:diff-mf}, we have
\begin{equation*}
\begin{split}
  &\abs{f(u_1)-f(u_2)} \\
  &\quad\leq  \sum_{i=1}^2\Big( c_\mu m_f(u_i,3\delta)+c_\mu \sum_{k=0}^K  m_f(u_i,2^{1-k}\delta)+\fint_{B(u_i,2^{-K}\delta)}\abs{f-f(u_i)}\ud\mu\Big).
\end{split}
\end{equation*}

Let $K$ be the unique positive integer such that $2^{-K}<\alpha\leq 2^{1-K}$. Then
\begin{equation*}
  \fint_{B(u_i,2^{-K}\delta)}\abs{f-f(u_i)}\ud\mu
  \leq 2^{-K}\delta\Norm{f}{\Lip}<\alpha\delta\Norm{f}{\Lip}.
\end{equation*}

On the other hand, for each $\theta\in\{3,2^{1-k}:k=0,1,\ldots,K\}\subset[\alpha,3]$, we have
\begin{equation*}
  m_f(u_i,\theta\delta)
  \leq c_P\theta\delta\Big(\fint_{B(u_i,\lambda\theta\delta)}(\lip f)^p\ud\mu\Big)^{1/p},
\end{equation*}
and hence
\begin{equation*}
\begin{split}
  m_f(u_i,3\delta) &+\sum_{k=0}^K m_f(u_i,2^{1-k}\delta) \\
  &\leq c_P\Big(3+\sum_{k=0}^K 2^{1-k}\Big)\delta \sup_{\alpha\lambda\delta\leq s\leq 3\lambda\delta}\Big(\fint_{B(u_i,s)}(\lip f)^p\ud\mu\Big)^{1/p},
\end{split}
\end{equation*}
where
\begin{equation*}
  3+\sum_{k=0}^K 2^{1-k}<3+\sum_{k=0}^\infty 2^{1-k}=3+4=7.
\end{equation*}
A combination of these estimates proves the lemma.
\end{proof}

We are now prepared to give:

\begin{proof}[Proof of Proposition \ref{prop:Lip-gf}]
Since the claim of the Proposition is local in nature, we may assume without loss of generality that $f$ is globally Lipschitz. Indeed, once the result is proved for globally Lipschitz functions, the locally Lipschitz case easily follows by replacing $f$ by $\phi f$, where $\phi$ is a Lipschitz function identically $1$ in a ball $B(x_0,2R)$ and vanishing outside $B(x_0,4R)$; then $\phi f$ is globally Lipschitz, and $g_{\phi f}(x)=g_f(x)$ and $\operatorname{lip}_{\phi f}(x)=\lip_f(x)$ for all $x\in B(x_0,R)$. Thus, for the rest of the proof, we assume that $f$ is globally Lipschitz.

The claim is trivial for points where $\lip f(x)=0$. Thus, for the rest of the proof, we only consider points with $\lip f(x)>0$. All such points belong to $A_k:=\{x\in X: 2^k\leq\lip f(x)<2^{k+1}\}$ for some $k\in\Z$. Almost every $x\in A_k$ is a density point of $A_k$. Hence a.e. $x\in X$ with $\lip f(x)>0$ is a density point of $\{y\in X:\lip f(y)\leq 2\lip f(x)\}$. We fix one such point $x$, and prove the claimed estimate. By definition of $g_f(x)$, we need to show that there exists some $r(x,f)>0$ such that, for all $r\in(0,r(x,f)]$,
\begin{equation*}
  \frac{1}{r}\fint_{B(x,r)}\abs{f-\ave{f}_{B(x,r)}}\ud\mu\gtrsim \lip f(x).
\end{equation*}

Let
\begin{equation*}
  \alpha:=\alpha(x,f):=\frac{\lip f(x)}{\Norm{f}{\Lip}}\in(0,1].
\end{equation*}
Since this depends only on $x$ and $f$, we may and will allow $r(x,f)$ to depend on $\alpha$. We also define another small parameter
\begin{equation}\label{eq:defTheta}
  \theta:=\frac{1}{16\max(\lambda,22C)},
\end{equation}
where $C=c_\mu\cdot c_P$ depends only on the parameters of the space.

For any $r>0$ and $z\in B(x,r)$, we note that $$B(x,r)\subset B(z,2r)=2(\theta\alpha)^{-1}B(z,\theta\alpha r),$$ and hence
\begin{equation*}
  \mu(B(x,r))\leq c(\mu,\theta\alpha)\mu(B(z,\theta\alpha r)),
\end{equation*}
where $c(\mu,\alpha)$ can be chosen to depend only on $\alpha=\alpha(x,f)$ and the doubling constant $c_\mu$.

Since $x$ is a density point of $\{y\in X:\lip f(y)\leq 2\lip f(x)\}$, there is a radius $r_1(x,f)>0$ 
 such that, for all $r\in(0,r_1(x,f)]$,
\begin{equation*}
  \frac{\mu(\{y\in B(x,r):\lip f(y)>2\lip f(x)\})}{\mu(B(x,r))}<\frac{\alpha^p}{c(\mu,\theta\alpha)}.
\end{equation*}
For such $r$, suppose that $B(z,s)\subset B(x,r)$ satisfies $s\in[\theta\alpha r,r]$. Then, by the previous two displayed estimates,
\begin{equation}\label{eq:lipMostlySmall}
\begin{split}
  &\frac{\mu(\{y\in B(z,s):\lip f(y)>2\lip f(x)\})}{\mu(B(z,s))} \\
  &\leq \frac{\mu(\{y\in B(x,r):\lip f(y)>2\lip f(x)\})}{\mu(B(z,\theta\alpha r))} \\
  &< \frac{\alpha^p}{c(\mu,\theta\alpha)} \frac{\mu(B(x,r))}{\mu(B(z,\theta\alpha r))} \leq \alpha^p.
\end{split}
\end{equation}
Thus, for all $B(z,s)\subset B(x,r)$ with $s\in[\theta\alpha r,r]$,
\begin{equation}\label{eq:lipIntSmall}
\begin{split}
  &\Big(\fint_{B(z,s)}(\lip f)^p\ud\mu\Big)^{1/p} \\
  &\leq\Big((2\lip f(x))^p+\frac{\mu(\{y\in B(z,s):\lip f(y)>2\lip f(x)\})}{\mu(B(z,s))}\Norm{f}{\Lip}^p\Big)^{\frac1p} \\
  &\leq\Big((2\lip f(x))^p+\alpha^p\Norm{f}{\Lip}^p\Big)^{\frac1p}  \\
  &\leq 2\lip f(x)+\alpha\Norm{f}{\Lip} 
  =3\lip f(x).
\end{split}
\end{equation}

By definition of $\lip f(x)$, there is $r_2(x,f)\in(0,r_1(x,f)]$ such that, for all $r\in(0,r_2(x,f)]$,
\begin{equation*}
  \sup_{\rho(y,x)\leq r}\frac{\abs{f(x)-f(y)}}{r}>\frac12\lip f(x).
\end{equation*}
In particular, with $\frac34 r$ in place of $r$, for every $r\in(0,r_2(x,f)]$ we can find some $w\in\bar B(x,\frac34 r)$ such that
\begin{equation*}
   \abs{f(x)-f(w)}>\frac 34 r\cdot\frac12 \lip f(x)=\frac38 r \lip f(x).
\end{equation*}

Let then $z\in\{x,w\}$ and $u\in B(z,\theta r)$. By Lemma \ref{lem:diffBound} with $\delta=\theta r$, we have
\begin{equation*}
\begin{split}
  \abs{f(u)-f(z)}
  &\leq 2\alpha\theta r\Norm{f}{\Lip}+14C\theta r
  \sup_{\substack{ v\in\{u,z\} \\ \alpha\lambda\theta r\leq s\leq 3\lambda\theta r }}
  \Big(\fint_{B(v,s)}(\lip f)^p\ud\mu\Big)^{\frac1p} \\
  &=:I+II,
\end{split}
\end{equation*}
where
\begin{equation*}
  I=2\theta r \lip f(x).
\end{equation*}
As for $II$, we would like to apply \eqref{eq:lipIntSmall}, so we need to check that each $B(v,s)$ appearing in $II$ satisfies the relevant assumptions. Since $\lambda\geq 1$, it is immediate that $s\geq\alpha\lambda\theta r\geq \theta\alpha r$, which is the required lower bound. Since $v\in\{u,z\}\subset B(z,\theta r)\subset B(z,\lambda\theta r)$ and $z\in\{x,w\}\subset\bar B(x,\frac34 r)$, it follows that
\begin{equation*}
  B(v,3\lambda\theta r)
  \subset B(z,\lambda\theta r+3\lambda\theta r)
  \subset B(x,\frac34 r+4\lambda\theta r)
  \subset B(x,r);
\end{equation*}
in the last step, we used $\theta\leq (16\lambda)^{-1}$, which is clear from \eqref{eq:defTheta}.
Then also $s\leq 3\lambda\theta r\leq \frac{3}{16}r\leq r$, which is the required upper bound.

Hence, we have
\begin{equation*}
  \sup_{\substack{ v\in\{u,z\} \\ \alpha\lambda\theta r\leq s\leq 3\lambda\theta r }}
  \Big(\fint_{B(v,s)}(\lip f)^p\ud\mu\Big)^{\frac1p}\leq 3\lip f(x)
\end{equation*}
and hence
\begin{equation*}
\begin{split}
  \abs{f(u)-f(z)}\leq I+II
  &\leq 2\theta r\lip f(x)+14C\theta r\cdot 3\lip f(x) \\
  &\leq 44 C\theta r\lip f(x)\leq \frac18 r\lip f(x);
\end{split}
\end{equation*}
in the last step, we used $\theta\leq (16\cdot 22C)^{-1}=(8\cdot 44 C)^{-1}$, which is immediate from the choice made in \eqref{eq:defTheta}.

Hence, we have proved that, for all $u\in B(x,\theta r)$ and $v\in B(z,\theta r)$, we have
\begin{equation*}
\begin{split}
  \abs{f(u)-f(v)}
  &\geq\abs{f(x)-f(z)}-\abs{f(u)-f(x)}-\abs{f(v)-f(z)} \\
  &\geq\frac38 r\lip f(x)-\frac18 r\lip f(x)-\frac18 r\lip f(x)= \frac18 r\lip f(x).
\end{split}
\end{equation*}
Hence
\begin{equation*}
\begin{split}
  \fint_{B(x,r)}\abs{f-\ave{f}_{B(x,r)}}\ud\mu
  &\geq \frac12 \fint_{B(x,r)}\fint_{B(x,r)}\abs{f(u)-f(v)}\ud\mu(u)\ud\mu(v) \\
  &\geq \frac12 \frac{\mu(B(x,\theta r))\mu(B(z,\theta r))}{\mu(B(x,r))^2}\frac18 r\lip f(x) \\
  &\geq c_1(\mu,\theta) r\lip f(x),
\end{split}
\end{equation*}
where $c_1(\mu,\theta)$ depends only the doubling constant of $\mu$ and the parameter $\theta$. Since $\theta$, as chosen in \eqref{eq:defTheta},
 depends only on the parameters of the space, so does $c_1(\mu,\theta)$. Moreover, the estimates above were obtained for all $r\geq r_1(x,f)$. Thus, we can take $\liminf$ to deduce that
\begin{equation*}
  \liminf_{r\to 0}\fint_{B(x,r)}\abs{f-\ave{f}_{B(x,r)}}\ud\mu\gtrsim \lip f(x),
\end{equation*}
where the implied constant depends only on the parameters of the space, as claimed.
\end{proof}

\section{Lipschitz approximation in $L^{p,\infty}(\nu_p)$}\label{sec:LipApprox}

Recall that the main result of the previous section, Proposition \ref{prop:Lip-gf}, was proved for locally Lipschitz functions only. In order to prove Theorem \ref{thm:Rupert} without such {\em a priori} assumptions, we will need to make use of some approximation results. The density of Lipschitz functions in Sobolev spaces is well known: see \cite[Theorem 8.2.1]{HKST} for their density in another Sobolev-type space $N^{1,p}$, and \cite[Corollary 10.2.9]{HKST} for the coincidence $M^{1,p}=N^{1,p}$, both under the same assumptions that we are using. However, we need Lipschitz approximation in the {\em a priori} larger (by Proposition \ref{prop:MpLowerBd}) space of functions with $\Norm{m_f}{L^{p,\infty}(\nu_p)}<\infty$, and such results do not seem to be available in the previous literature, except in the special cases of $\R^d$ \cite{Frank} and Carnot groups \cite{LXY}, where they are relatively easy consequences of standard mollification by convolution. 

In the setting at hand, we already introduced the approximate identity $\Phi_t$ in \eqref{eq:Phi-tf}, and proved some first properties in Lemma \ref{lem:PhitDiff}. Let us next compare the mean oscillations of $f$ and $\Phi_t f$:

\begin{lemma}\label{lem:oscPhit}
Let $(X,\rho,\mu)$ be a doubling metric measure space. Then for all functions $f\in L^1_{\loc}(\mu)$, points $z\in X$, and $r,t>0$,
the approximations $\Phi_t f$ from \eqref{eq:Phi-tf} satisfy the estimate
\begin{equation}\label{eq:oscPhit}
  m_{\Phi_t f}(z,r)\lesssim \frac{\min(r,t)}{t} m_f(z,c\max(r,t))
\end{equation}
for some $c$ that only depends on the parameters of the space.
\end{lemma}

\begin{proof}
If $r\leq t$, then all $x,y\in B(z,r)$ satisfy $\rho(x,y)\leq 2r\leq 2t$ and $B(x,b t)\subset B(z,r+bt)\subset B(z,(1+b)t)\subset B(x,(2+b)t)$
Thus we obtain from Lemma \ref{lem:PhitDiff} that
\begin{equation*}
\begin{split}
  \abs{\Phi_t f(x)-\Phi_t f(y)}
  &\lesssim\frac{\rho(x,y)}{t}\fint_{B(x,b t)}\abs{f-\ave{f}_{B(x,bt)}}\ud\mu \\
  &\lesssim\frac{r}{t}\fint_{B(z,(1+b)t)}\abs{f-\ave{f}_{B(z,(1+b)t)}}\ud\mu.
\end{split}
\end{equation*}
Hence
\begin{equation*}
\begin{split}
  \operatorname{LHS}\eqref{eq:oscPhit}
  &\leq\fint_{B(z,r)}\fint_{B(z,r)}\abs{\Phi_t f(x)-\Phi_t f(y)}\ud\mu(x)\ud\mu(y) \\
  &\lesssim\frac{r}{t}\fint_{B(x,(1+b) t)}\abs{f-\ave{f}_{B(x,(1+b)t)}}\ud\mu.
\end{split}
\end{equation*}

We turn to $r\geq t$. It suffices to prove the estimate with some $k$ in place of $\ave{\Phi_t f}_{B(z,r)}$ in the expression
\begin{equation*}
  m_{\Phi_t}(z,r)=\fint_{B(z,r)} \abs{\Phi_t f-\ave{\Phi_t f}_{B(z,r)}}\ud\mu.
\end{equation*}
We have
\begin{equation*}
  \Phi_t f-k
  =\sum_i \phi_i(\ave{f}_{B(x_i,t)}-k)
  =\sum_i \phi_i\ave{f-k}_{B(x_i,t)}.
\end{equation*}
Hence
\begin{equation*}
  \fint_{B(z,r)}\abs{\Phi_t f-k}\ud\mu
  \leq\sum_{\substack{ i:B(x_i,(1+\gamma) t) \\ \phantom{i:}\cap B(z,r)\neq\varnothing}}
  \frac{\mu(B(x_i,(1+\gamma) t))}{\mu(B(z,r)}\fint_{B(x_i,t)}\abs{f-k}\ud\mu.
\end{equation*}
If $B(x_i,(1+\gamma) t)\cap B(z,r)\neq\varnothing$, then $$B(x_i,t)\subset B(z,r+2(1+\gamma) t)\subset B(z,(3+2\gamma)r).$$

Abbreviating the summation condition $B(x_i,(1+\gamma) t)\cap B(z,r)\neq\varnothing$ as ``$*$'', it follows that
\begin{equation*}
\begin{split}
  \fint_{B(z,r)}\abs{\Phi_t f-k}\ud\mu
  &\lesssim \frac{1}{\mu(B(z,r))}\int_{B(z,(3+2\gamma) r)}\sum_i^* 1_{B(x_i,t)}\abs{f-k}\ud\mu \\
   &\lesssim \frac{1}{\mu(B(z,r))}\int_{B(z,(3+2\gamma) r)}\abs{f-k}\ud\mu \\
   &\lesssim \fint_{B(z,(3+2\gamma) r)}\abs{f-k}\ud\mu,
\end{split}
\end{equation*}
using in the second step the bounded overlap of the balls $B(x_i,t)$. Now we are free to pick $k=\ave{f}_{B(z,r+3t)}$.

Comparing the bounds in the two cases $r\leq t$ and $r\geq t$, and applying doubling one more time, we obtain the claimed bound \eqref{eq:oscPhit} with $c=\max\{1+b,3+2\gamma\}$.
\end{proof}

A weakness of Lemma \ref{lem:oscPhit} is that, for $r<t$, oscillations of $\Phi_t f$ on the small scale $r$ are dominated by the oscillations of $f$ on the large scale $t$. This issue will be rectified with the help of the macroscopic Poincar\'e inequality (Theorem \ref{thm:macroPoincare}) in the following lemma, which allows us to control the mean oscillations of $\Phi_t f$ by the mean oscillations of $f$ on the same scale:

\begin{lemma}\label{lem:mPhiVsmf}
Let $(X,\rho,\mu)$ be a doubling metric measure space and satisfy the $(1,p)$-Poincar\'e inequality with constants $(\lambda,c_P)$.  Then for all $f\in L^1_{\loc}(\mu)$, all $z\in X$ and $r,t>0$, the approximations $\Phi_t f$ from \eqref{eq:Phi-tf} satisfy the estimate
\begin{equation*}
  m_{\Phi_t f}(z,r)
  \lesssim \Big(\fint_{B(z,c\tilde\lambda\max(r,t))} m_f(x,r)^p\ud\mu(x)\Big)^{\frac1p}.
\end{equation*}
\end{lemma}

Thus, the oscillation of the mollification $\Phi_t f$ are controlled by an average of oscillations of the original $f$ at the same scale $r$. This is our substitute for the Euclidean estimate \cite[proof of Lemma 3.2]{Frank}
\begin{equation}\label{eq:Frank-mphif}
  m_{\varphi*f}(z,r)\leq
  \int_{\R^d}\varphi(x)m_f(z-x,r)\ud z,
\end{equation}
which follows easily in the presence of an underlying group structure.

\begin{proof}
By Lemma \ref{lem:oscPhit} in the first step and Theorem \ref{thm:macroPoincare} in the second, we obtain
\begin{equation*}
\begin{split}
  m_{\Phi_t f}(z,r)
  &\lesssim \frac{\min(r,t)}{t}m_f(z,c\max(r,t)) \\
  &\lesssim \frac{\min(r,t)}{t}\frac{\max(r,t)}{r}\Big(\fint_{B(z,c\tilde\lambda\max(r,t))} m_f(x,r)^p\ud\mu(x)\Big)^{\frac1p} \\
  &=\Big(\fint_{B(z,c\tilde\lambda\max(r,t))} m_f(x,r)^p\ud\mu(x)\Big)^{\frac1p},
\end{split}
\end{equation*}
where in the last step we simply observed that $\min(r,t)\max(r,t)=rt$.
\end{proof}

We can now show that the mollifications $\Phi_t f$ remain uniformly bounded in the $L^{p,\infty}(\nu_p)$ norm:

\begin{lemma}\label{lem:mPhitf-mf}
Let $p\in(1,\infty)$ and $(X,\rho,\mu)$ be a complete doubling metric measure space and satisfy the $(1,p)$-Poincar\'e inequality with constants $(\lambda,c_P)$.  Then for all $f\in L^1_{\loc}(\mu)$ and all $t>0$, the approximations $\Phi_t f$ from \eqref{eq:Phi-tf} satisfy the estimate
\begin{equation*}
  \Norm{m_{\Phi_t f}}{L^{p,\infty}(\nu_p)}\lesssim   \Norm{m_f}{L^{p,\infty}(\nu_p)}.
\end{equation*}
\end{lemma}

This is our analogue of \cite[Lemma 3.2]{Frank}, using Lemma \ref{lem:mPhiVsmf} as a substitute of \eqref{eq:Frank-mphif}. In the proof of \cite[Lemma 3.2]{Frank}, using a proper norm equivalent to the quasi-norm $\Norm{\ }{L^{p,\infty}}$ and the translation invariance of the Lebesgue measure, it is immediate that
\begin{equation*}
\begin{split}
  \Norm{m_{\varphi*f}}{L^{p,\infty}(\nu_p)}
  &\lesssim\int_{\R^d}\varphi(x)\Norm{m_f(\cdot-x,\cdot)}{L^{p,\infty}(\nu_p)}\ud x \\
  &=\int_{\R^d}\varphi(x)\Norm{m_f}{L^{p,\infty}(\nu_p)}\ud x
  =\Norm{\varphi}{L^1}\Norm{m_f}{L^{p,\infty}(\nu_p)}.
\end{split}
\end{equation*}

\begin{proof}
Slightly more generally, we will prove the lemma for any product $\mu\times\tau$ of the measure $\mu$ on $X$ and any $\sigma$-finite measure $\tau$ on $(0,\infty)$; in the case of $\nu_p$, we have $\ud\tau(r)=r^{-p-1}\ud r$. Writing $\mu\times\tau$ instead of $\nu_p$ will hopefully clarify the argument, since we will make some use of interpolation with respect to the Lebesgue exponent $p$, but no change of measure will be involved.

By the Keith--Zhong Theorem \ref{thm:KZ}, $(X,\rho,\mu)$ also satisfies a $(1,p-\eps)$-Poincar\'e inequality for some $\eps\in(0,p-1)$. Thus, by Lemma \ref{lem:mPhiVsmf}, we have
\begin{equation*}
\begin{split}
 &\Norm{m_{\Phi_t f}}{L^{p,\infty}(\mu\times\tau)} \\
 &\lesssim\BNorm{(z,r)\mapsto\Big(\fint_{B(z,c\tilde\lambda\max(r,t))}m_f(x,r)^{p-\eps}\ud\mu(x)\Big)^{1/(p-\eps)}}{L^{p,\infty}(\mu\times\tau)} \\
 &=\BNorm{(z,r)\mapsto \fint_{B(z,c\tilde\lambda\max(r,t))}m_f(x,r)^{p-\eps}\ud\mu(x)}{L^{p/(p-\eps),\infty}(\mu\times\tau)}^{1/(p-\eps)}.
\end{split}
\end{equation*}
Regarding $t$ as fixed throughout the argument, let us denote
\begin{equation*}
  q:=\frac{p}{p-\eps}>1,\qquad g:=m_f^{p-\eps},\qquad\alpha(r):=c\tilde\lambda\max(r,t),
\end{equation*}
 and
\begin{equation*}
  \mathcal Ah(z,r):=\fint_{B(z,\alpha(r))} g(x,r)\ud\mu(x).
\end{equation*}
The claimed bound will follow if we can prove that
\begin{equation}\label{eq:weakType2prove}
  \Norm{\mathcal Ag}{L^{q,\infty}(\mu\times\tau)}\lesssim\Norm{g}{L^{q,\infty}(\mu\times\tau)},\qquad q>1.
\end{equation}
Let us first consider the corresponding strong-type bound. By Fubini's theorem, we have
\begin{equation*}
\begin{split}
  \Norm{\mathcal Ag}{L^q(\mu\times\tau)}^q
  =\int_0^\infty\int_X\Babs{\fint_{B(z,\alpha(r))} g(x,r)\ud\mu(x)}^q\ud\mu(z)\ud\tau(r),
\end{split}
\end{equation*}
where
\begin{equation*}
\begin{split}
  \int_X &\Babs{\fint_{B(z,\alpha(r))} g(x,r)\ud\mu(x)}^q\ud\mu(z) \\
   &\leq \int_X\fint_{B(z,\alpha(r))} \abs{g(x,r)}^q\ud\mu(x)\ud\mu(z) \\
    &=\int_X\int_X\frac{1_{\{\rho(x,z)<\alpha(r)\}}}{\mu(B(z,\alpha(r))}\abs{g(x,r)}^q\ud\mu(x)\ud\mu(z) \\
    &\eqsim\int_X\Big(\int_X\frac{1_{\{\rho(x,z)<\alpha(r)\}}}{\mu(B(x,\alpha(r))}\ud\mu(z)\Big) \abs{g(x,r)}^q\ud\mu(x)
    =\int_{X}\abs{g(x,r)}^q\ud\mu(x).
\end{split}
\end{equation*}
Integrating over $r$, it follows that
\begin{equation}\label{eq:strongTypeProved}
  \Norm{\mathcal Ag}{L^q(\mu\times\tau)}\lesssim
  \Big(\int_0^\infty \int_{X}\abs{g(x,r)}^q\ud\mu(x)\ud\tau(r)\Big)^{1/q}=\Norm{g}{L^q(\mu\times\tau)}.
\end{equation}
This estimate works for all $q\in[1,\infty]$; the case $q=\infty$ (that we do not really need but just mention for completeness) requires an easy modification of the argument.

This is the strong type bound corresponding to \eqref{eq:weakType2prove}. The required \eqref{eq:weakType2prove} then follows by choosing some $1\leq q_0<q<q_1\leq\infty$ and interpolating \eqref{eq:strongTypeProved} between these two values. See e.g.~\cite[Theorem 2.2.3]{HNVW1} for a version covering the case at hand, with an $L^{q,\infty}$-to-$L^{q,\infty}$-bound \eqref{eq:weakType2prove} as a conclusion.
\end{proof}

\section{The upper bound for the Sobolev norm}\label{sec:MpUpperBd}

We finally return to the other direction of the proof of Theorem \ref{thm:Rupert}. 
With the results of Sections \ref{sec:Lip-gf} and \ref{sec:LipApprox} available as substitutes of some Euclidean tools, the proof can now be completed by an adaptation of the Euclidean argument of \cite{Frank}. The following lemma borrows some ideas from the proof of \cite[Lemma 3.1]{Frank}, but is actually simpler. This may be partly attributed to the fact that we only prove equivalences of norms, in contrast to the identities achieved in the presence of Euclidean symmetries in \cite{Frank}.

\begin{lemma}\label{lem:gfp<}
Let $(X,\rho,\mu)$ be a space of homogeneous type, and $f\in L^1_{\loc}(\mu)$. Then
\begin{equation*}
   \int_{X}g_f^p\ud\mu\leq  p\liminf_{\kappa\to 0}\kappa^p\nu_p(\{m_f>\kappa\}),
\end{equation*}
where $g_f$ and $m_f$ are as in \eqref{eq:gf}.
\end{lemma}

\begin{proof}
Let $\eps>0$ and, for $j\in\N$,
\begin{equation*}
  E_j:=\{x\in B(x_0,2^j):\text{ for all }t\in(0,2^{-j}]: g_f(x)\leq (1+\eps)\frac1tm_f(x,t)\}.
\end{equation*}
Then $E_j\subseteq E_{j+1}$ and it follows from the definition of $\liminf$ that $\bigcup_{j=0}^\infty E_j=X$.

\begin{equation*}
\begin{split}
   \int_{E_j} g_f(x)^p\ud\mu(x)
   &=\kappa^p\int_{E_j} \Big(\frac{g_f(x)}{\kappa}\Big)^p\ud\mu(x) \\
   &=\kappa^p\int_{E_j} [\Big(\frac{g_f(x)}{\kappa}\Big)^p-2^{jp}]\ud\mu(x)+\kappa^p2^{jp}\mu(E_j) =:I+II,
\end{split}
\end{equation*}
where
\begin{equation*}
\begin{split}
  I  & \leq \kappa^p\int_{E_j} [\Big(\frac{g_f(x)}{\kappa}\Big)^p-2^{jp}]_+\ud\mu(x) \\
  &=\kappa^p\int_{E_j}\Big(\int_{\kappa/g_f(x)}^{2^{-j}} p t^{-p-1}\ud t\Big)_+\ud\mu(x) \\
  &=\kappa^p\int_{E_j}\int_0^{2^{-j}} 1_{\{(x,t): tg_f(x)>\kappa\}} p t^{-p-1}\ud t\ud\mu(x) \\
  &\leq\kappa^p\int_{E_j}\int_0^{2^{-j}} 1_{\{(x,t): (1+\eps)m_f(x,t)>\kappa\}} p t^{-p-1}\ud t\ud\mu(x) \\
  &\leq p\kappa^p\nu_p\Big(\Big\{(x,t)\in X\times(0,\infty): m_f(x,t)>\frac{\kappa}{1+\eps}\Big\}\Big).
\end{split}
\end{equation*}
Thus, for any $\kappa>0$, we have
\begin{equation*}
  \int_{E_j}g_f^p\ud\mu\leq p\kappa^p\nu_p(\{m_f>\frac{\kappa}{1+\eps}\})+\kappa^p 2^{jp}\mu(E_j),
\end{equation*}
where the left-hand side is independent of $\kappa$. We can hence take $\liminf_{\kappa\to 0}$ on the right. In this limit, the final term vanishes, and hence we are left with
\begin{equation*}
\begin{split}
  \int_{E_j}g_f^p\ud\mu
  &\leq p\liminf_{\kappa\to 0}\kappa^p\nu_p(\{m_f>\frac{\kappa}{1+\eps}\}) \\
  &= p(1+\eps)^p\liminf_{\kappa\to 0}\kappa^p\nu_p(\{m_f>\kappa\}).
\end{split}
\end{equation*}
The right-hand side is independent of $j$, and hence we can let $j\to\infty$ on the left and apply monotone convergence to deduce that
\begin{equation*}
   \int_{X}g_f^p\ud\mu\leq  p(1+\eps)^p\liminf_{\kappa\to 0}\kappa^p\nu_p(\{m_f>\kappa\}).
\end{equation*}
Finally, letting $\eps\to 0$, we obtain the 
asserted bound.
\end{proof}

We can now prove the upper bound of the Sobolev norm in Theorem \ref{thm:Rupert}.

\begin{proposition}\label{prop:MpUpperBd}
Let $p\in(1,\infty)$ and let $(X,\rho,\mu)$ be a complete doubling metric space that supports the $(1,p)$-Poincar\'e inequality. If $f\in L^1_{\loc}(\mu)$ satisfies $m_f\in L^{p,\infty}(\nu_p)$, then $f\in\dot M^{1,p}(\mu)$ and
\begin{equation*}
  \Norm{f}{\dot M^{1,p}(\mu)}\lesssim\Norm{m_f}{L^{p,\infty}(\nu_p)}.
\end{equation*}
\end{proposition}

\begin{proof}
We approximate $f$ by the locally Lipschitz functions $\Phi_t f$ from \eqref{eq:Phi-tf}.
These satisfy
\begin{equation}\label{eq:lipPhi<mf}
\begin{split}
  \Norm{\lip(\Phi_t f)}{L^p(\mu)}
  &\lesssim\Norm{g_{\Phi_t f}}{L^p(\mu)}\qquad\text{(Proposition \ref{prop:Lip-gf})} \\
  &\lesssim\Norm{m_{\Phi_t f}}{L^{p,\infty}(\nu_p)}\qquad\text{(Lemma \ref{lem:gfp<})} \\
  &\lesssim\Norm{m_{f}}{L^{p,\infty}(\nu_p)}\qquad\text{(Lemma \ref{lem:mPhitf-mf})}.
\end{split}
\end{equation}
By the Keith--Zhong Theorem \ref{thm:KZ}, $(X,\rho,\mu)$ satisfies the $(1,p-\eps)$-Poincar\'e inequality for some $\eps>0$. By estimate \eqref{eq:lipPhi<mf}, the functions $(\lip(\Phi_t f))^{p-\eps}$ form a bounded set in the reflexive space $L^{\frac{p}{p-\eps}}(\mu)$. Hence we can find a subsequence $(\lip(\Phi_{t_j} f))^{p-\eps}$, with $t_j\to 0$, that converges weakly in $L^{\frac{p}{p-\eps}}(\mu)$ to some function $h$, so that $g:=h^{\frac{1}{p-\eps}}\in L^p(\mu)$. From the properties of weak limits, we have
\begin{equation}\label{eq:g-mf}
\begin{split}
  \Norm{g}{L^p(\mu)}
  =\Norm{h}{L^{\frac{p}{p-\eps}}(\mu)}^{\frac{1}{p-\eps}} 
  &\leq\liminf_{j\to\infty}\Norm{(\lip(\Phi_{t_j} f))^{p-\eps}}{L^{\frac{p}{p-\eps}}(\mu)}^{\frac{1}{p-\eps}} \\
  &=\liminf_{j\to\infty}\Norm{\lip(\Phi_{t_j} f)}{L^p(\mu)} 
  \lesssim\Norm{m_f}{L^{p,\infty}(\nu_p)}.
\end{split}
\end{equation}

Since also $\Phi_{t_j} f\to f$ in $L^1_{\loc}(\mu)$, it follows that
\begin{equation*}
\begin{split}
    \frac{1}{r(B)} &\fint_B\abs{f-\ave{f}_B} \\
    &=\lim_{j\to\infty}\frac{1}{r(B)}\fint_B\abs{\Phi_{t_j}f-\ave{\Phi_{t_j}f}_B}\quad\text{since }\Phi_{t_j}f\to f\text{ in }L^1_{\loc}(\mu) \\
    &\lesssim\liminf_{j\to\infty}\Big(\fint_{\lambda B}(\lip\Phi_{t_j}f)^{p-\eps}\Big)^{\frac{1}{p-\eps}}\quad\text{by $(1,p-\eps)$-Poincar\'e} \\
    &=\Big(\fint_{\lambda B} g^{p-\eps}\Big)^{\frac{1}{p-\eps}}
\end{split}  
\end{equation*}
by the weak convergence $(\lip\Phi_{t_j}f)^{p-\eps}\rightharpoonup g^{p-\eps}$ in $L^{\frac{p}{p-\eps}}(\mu)$ and the observation that $1_{\lambda B}\in L^{\frac{p}{\eps}}(\mu)=(L^{\frac{p}{p-\eps}}(\mu))^*$. By Lemma \ref{lem:diff-mf}, we have, for all Lebesgue points $u_1,u_2\in X$ of $f$ at distance $\delta=\rho(u_1,u_2)$,
\begin{equation}\label{eq:diff-Mg}
\begin{split}
  \abs{f(u_1)-f(u_2)}
  &\leq  c_\mu \sum_{i=1}^2\Big[ m_f(u_i,3\delta)+\sum_{k=0}^\infty  m_f(u_i,2^{1-k}\delta)\Big] \\
  &\lesssim \sum_{i=1}^2\Big[ 3\delta\Big(\fint_{B(u_i,3\lambda\delta)} g^{p-\eps}\Big)^{\frac{1}{p-\eps}} \\
  &\qquad +\sum_{k=0}^\infty 2^{1-k}\lambda\delta \Big(\fint_{B(u_i,2^{1-k}\lambda\delta)} g^{p-\eps}\Big)^{\frac{1}{p-\eps}}\Big] \\
  &\leq\sum_{i=1}^2\Big(3+\sum_{k=0}^\infty 2^{1-k}\Big)\delta (M_{p-\eps}g)(u_i) \\
  &=7\rho(u_1,u_2)\sum_{i=1}^2 (M_{p-\eps}g)(u_i),
\end{split}
\end{equation}
where $M_{p-\eps}g:=[M(g^{p-\eps})]^{\frac{1}{p-\eps}}$ is the rescaled maximal function.

This shows that $c\cdot M_{p-\eps}g$, for some constant $c$ depending only on the parameters of the space, is a Haj\l asz upper gradient of $f$, and hence
\begin{equation*}
\begin{split}
  \Norm{f}{\dot M^{1,p}(\mu)}
  &\lesssim\Norm{M_{p-\eps}g}{L^p(\mu)}\qquad\text{(definition of $\dot M^{1,p}(\mu)$)} \\
  &\lesssim\Norm{g}{L^p(\mu)}\qquad\text{(maximal inequality)} \\
  &\lesssim\Norm{m_f}{L^{p,\infty}(\nu_p)}\qquad\text{(by \eqref{eq:g-mf})}.
\end{split}
\end{equation*}
This completes the proof.
\end{proof}

\section{The limit relation}\label{sec:limit}

Combining Propositions \ref{prop:MpLowerBd} and \ref{prop:MpUpperBd}, we obtain the norm equivalence \eqref{eq:CST} of Theorem \ref{thm:Rupert}. To complete the proof of the said theorem, it remains to establish the limit relation \eqref{eq:Rupert}. Since $\liminf$ is obviously bounded by $\sup$, we only need to consider the estimate ``$\lesssim$''.

We will need the following density result, which is most likely known. Since most sources deal with the non-homogeneous space $M^{1,p}$ instead of $\dot M^{1,p}$ (see e.g. \cite[Lemma 10.2.7]{HKST}), we provide the short proof for completeness.

\begin{proposition}\label{prop:LipDense}
Let $(X,\rho,\mu)$ be a doubling metric measure space.
Then Lip\-schitz functions are dense in $\dot M^{1,p}(\mu)$.
\end{proposition}

\begin{proof}
Let $h\in L^p(\mu)$ be a Haj\l asz upper gradient of $f\in\dot M^{1,p}$. Given $L>0$, consider the restriction $f|\{h\leq L\}$. For $x,y\in\{h\leq L\}$, we have
\begin{equation*}
  \abs{f(x)-f(y)}\leq \rho(x,y)(h(x)+h(y))\leq 2L \rho(x,y),
\end{equation*}
so this restriction is $2L$-Lipschitz. Hence it can be extended to a $cL$-Lipschitz function $\tilde f$ on all of $X$. (For real-valued functions, the classical McShane extension \cite[page 99]{HKST} preserves the Lipschitz constant, so we may take $c=2$; for $\C\simeq\R^2$ -valued functions, $c=2\sqrt{2}$ will fo \cite[Corollary 4.1.6]{HKST}. In fact, this extension is possible even for Banach space $V$-valued functions, but then the constant $c$ may also depend on the parameters of the domain space $X$ \cite[Theorem 4.1.12]{HKST}.)

Now consider $f':=f-\tilde f$. We claim that 
\begin{equation}\label{eq:Haj2show}
   \abs{f'(x)-f'(y)}\lesssim \rho(x,y)\max_{z\in\{x,y\}}h(z)1_{\{h(z)>L\}},
\end{equation}
which shows that $h':=c h\cdot 1_{\{h>L\}}$ is a Haj\l asz upper gradient of $f'$. If $x,y\in\{h\leq L\}$, then \eqref{eq:Haj2show} holds trivially in the form $0\leq 0$. Otherwise, we always have
\begin{equation*}
\begin{split}
  &\abs{f'(x)-f'(y)}
  \leq\abs{f(x)-f(y)}+\abs{\tilde f(x)-\tilde f(y)} \\
  &\quad\leq \rho(x,y)(h(x)+h(y))+\rho(x,y) cL
  =\rho(x,y)(h(x)+h(y)+cL),
\end{split}
\end{equation*}
If at least one of $x,y$ belongs to $\{h>L\}$, then $\max\{h(x),h(y),L\}=h(z)>L$ for some $z\in\{x,y\}$. Thus
\begin{equation*}
  h(x)+h(y)+cL\leq (2+c)h(z)=(2+c)h(z)1_{\{h(z)>L\}},
\end{equation*}
as required for \eqref{eq:Haj2show}. Thus
\begin{equation*}
  \Norm{f-\tilde f}{\dot M^{1,p}}\lesssim\Norm{h\cdot 1_{\{h>L\}}}{L^p(\mu)},
\end{equation*}
and the right-hand side converges to zero as $L\to\infty$ by dominated convergence.
\end{proof}

We can now prove the following result, which provides the last missing piece \eqref{eq:Rupert} of Theorem \ref{thm:Rupert}.

\begin{proposition}\label{prop:RupertLimit}
Let $(X,\rho,\mu)$ be a complete doubling metric measure space that supports the $(1,p)$-Poincar\'e inequality.
For $f\in\dot M^{1,p}(\mu)$, we have
\begin{equation*}
  \Norm{f}{\dot M^{1,p}(\mu)}\lesssim\liminf_{\kappa\to 0}\kappa\nu_p(\{m_f>\kappa\})^{1/p}.
\end{equation*}
\end{proposition}

\begin{proof}
By Proposition \ref{prop:LipDense}, we can find a sequence of Lipschitz functions $f_n$ such that $\Norm{f-f_n}{\dot M^{1,p}(\mu)}\to 0$.
We start with
\begin{equation}\label{eq:limStart}
  \Norm{f}{\dot M^{1,p}}
  \leq\Norm{f_n}{\dot M^{1,p}}+\Norm{f-f_n}{\dot M^{1,p}},\qquad f_n\in\Lip.
\end{equation}
By the Keith--Zhong Theorem \ref{thm:KZ}, the space supports the $(1,p-\eps)$-Poincar\'e inequality for some $\eps>0$, and this says that
\begin{equation*}
  m_{f_n}(x,t)\lesssim t\Big(\fint_{B(x,\lambda t)}(\lip f_n)^{p-\eps}\ud\mu\Big)^{\frac{1}{p-\eps}}.
\end{equation*}
Then, from a computation like \eqref{eq:diff-Mg}, we see that $c M_{p-\eps}(\lip f_n)$ is a Haj\l asz upper gradient of $f_n$, and hence
\begin{equation}\label{eq:fnMp<lim}
\begin{split}
  \Norm{f_n}{\dot M^{1,p}}
  &\lesssim\Norm{M_{p-\eps}(\lip f_n)}{L^p}\qquad\text{(definition of $\dot M^{1,p}$)} \\
  &\lesssim\Norm{\lip f_n}{L^p}\qquad\text{(maximal inequality)} \\
  &\lesssim\Norm{g_{f_n}}{L^p}\qquad\text{(Proposition \ref{prop:Lip-gf})} \\
  &\lesssim\liminf_{\kappa\to 0}\kappa\nu_p(\{m_{f_n}>\kappa\})^{1/p}\qquad\text{(Lemma \ref{lem:gfp<})}.
\end{split}
\end{equation}
Now $m_{f_n}\leq m_f+m_{f-f_n}$ and hence
\begin{equation*}
  \kappa^p\nu_p(\{m_{f_n}>\kappa\})
  \leq \kappa^p\nu_p(\{m_{f}>\kappa/2\})
  +\kappa^p\nu_p(\{m_{f-f_n}>\kappa/2\}),
\end{equation*}
where
\begin{equation*}
\begin{split}
  \kappa^p\nu_p(\{m_{f-f_n}>\kappa/2\})
  &\lesssim\Norm{m_{f-f_n}}{L^{p,\infty}(\nu_p)}^p\qquad\text{(definition of $L^{p,\infty}$ norm)} \\ 
  &\lesssim\Norm{f-f_n}{\dot M^{1,p}(\mu)}^p\qquad\text{(Proposition \ref{prop:MpLowerBd})}.
\end{split}
\end{equation*}
Thus
\begin{equation}\label{eq:mfn<mf+pert}
  \liminf_{\kappa\to 0}\kappa\nu_p(\{m_{f_n}>\kappa\})^{1/p}
  \lesssim \liminf_{\kappa\to 0}\kappa\nu_p(\{m_{f}>\kappa\})^{1/p}+\Norm{f-f_n}{\dot M^{1,p}(\mu)}^p.
\end{equation}
Combining the estimates, we obtain
\begin{equation*}
\begin{split}
  \Norm{f}{\dot M^{1,p}(\mu)}
  &\leq\Norm{f_n}{\dot M^{1,p}}+\Norm{f-f_n}{\dot M^{1,p}}\qquad\text{(by \eqref{eq:limStart})}\\
  &\lesssim\liminf_{\kappa\to 0}\kappa\nu_p(\{m_{f_n}>\kappa\})^{1/p}+\Norm{f-f_n}{\dot M^{1,p}}
  \qquad\text{(by \eqref{eq:fnMp<lim})}\\
  &\lesssim\liminf_{\kappa\to 0}\kappa\nu_p(\{m_{f}>\kappa\})^{1/p}+\Norm{f-f_n}{\dot M^{1,p}}
  \qquad\text{(by \eqref{eq:mfn<mf+pert})}.
\end{split}
\end{equation*}
Only the last term depends on $n$, so taking the limit $n\to\infty$ and noting that the last term vanishes in this limit by the choice of $f_n$, we obtain the claimed estimate.
\end{proof}

\section{Further remarks on the Poincar\'e inequality}\label{sec:final}

In the proof of our main Theorem \ref{thm:Rupert}, the Poincar\'e inequality was used both in its usual infinitesimal form and via the new macroscopic version derived from it in Theorem \ref{thm:macroPoincare}. In this final section, we complement this discussion by collecting some new equivalent characterisations of the Poincar\'e inequality.

\begin{proposition}\label{prop:newPoincare2}
Let $(X,\rho,\mu)$ be a doubling metric measure space, and $p\in[1,\infty)$. Then the following conditions are equivalent:
\begin{enumerate}[\rm(1)]
  \item\label{it:usualPP} $(X,\rho,\mu)$ supports the $(1,p)$-Poincar\'e inequality in the sense of Definition \ref{def:Poincare};
  \item\label{it:gfPP} all $f\in\Lip(X)$ satisfy \eqref{eq:uPoincare} with $g=g_f$, where $g_f$ is defined as in \eqref{eq:gf}.
\end{enumerate}
\end{proposition}

\begin{proof}
\eqref{it:usualPP}$\Rightarrow$\eqref{it:gfPP}: Under the assumption \eqref{it:usualPP}, Proposition \ref{prop:Lip-gf} guarantees the pointwise ($\mu$-a.e.) inequality $\lip f(x)\lesssim g_f(x)$. Combining this with a direct application of \eqref{it:usualPP}, it follows that
\begin{equation*}
  m_f(z,s)\lesssim s\Big(\fint_{B(z,\lambda s)}(\operatorname{lip}f)^p\ud\mu\Big)^{\frac1p}
  \lesssim s\Big(\fint_{B(z,\lambda s)}g_f^p\ud\mu\Big)^{\frac1p}.
\end{equation*}
This is exactly \eqref{it:gfPP}, and completes the proof of \eqref{it:usualPP}$\Rightarrow$\eqref{it:gfPP}.

\eqref{it:gfPP}$\Rightarrow$\eqref{it:usualPP}: This is immediate from the elementary pointwise bound $g_f(x)\leq 2\lip f(x)$ from Remark \ref{rem:gf-Lip}.
\end{proof}

\begin{proposition}\label{prop:newPoincare}
Let $(X,\rho,\mu)$ be a doubling metric measure space, and $p\in[1,\infty)$. Then the following conditions are equivalent:
\begin{enumerate}[\rm(1)]
  \item\label{it:usualP} $(X,\rho,\mu)$ supports the $(1,p)$-Poincar\'e inequality in the sense of Lemma \ref{lem:Poincares}\eqref{it:Lipf}, i.e., with an integral of $\Lip f$ on the right-hand side;
  \item\label{it:macroP} all $f\in L^1_{\loc}(\mu)$ satisfy the conclusion of Theorem \ref{thm:macroPoincare};
  \item\label{it:macroPL} all $f\in \Lip(X)$ satisfy the conclusion of Theorem \ref{thm:macroPoincare}.
\end{enumerate}
If, in addition, $X$ is complete, then all these are further equivalent to the $(1,p)$-Poincar\'e inequality in the sense of Definition \ref{def:Poincare}.
\end{proposition}

\begin{proof}
\eqref{it:usualP}$\Rightarrow$\eqref{it:macroP} is Theorem \ref{thm:macroPoincare}, taking into account Remark \ref{rem:macroPoincare} that only the version of the $(1,p)$-Poincar\'e inequality as in Lemma \ref{lem:Poincares}\eqref{it:Lipf}, with $\Lip f$ on the right, is actually used in the proof of Theorem \ref{thm:macroPoincare}.

\eqref{it:macroP}$\Rightarrow$\eqref{it:macroPL} is trivial.

\eqref{it:macroPL}$\Rightarrow$\eqref{it:usualP}:
Let $f\in\Lip(X)$. By assumption \eqref{it:macroPL}, we have
\begin{equation*}
  \frac{m_f(z,s)}{s}\leq \tilde c_P\Big(\fint_{B(z,\tilde\lambda s)} \Big(\frac{m_f(x,r)}{r}\Big)^p\ud\mu(x)\Big)^{\frac1p},
\end{equation*}
and hence
\begin{equation*}
  \fint_{B(z,\tilde\lambda s)}\Big[(2\Norm{f}{\Lip})^p-\Big(\frac{m_f(x,r)}{r}\Big)^p\Big]\ud\mu(x)
  \leq(2\Norm{f}{\Lip})^p-\Big(\frac{m_f(z,s)}{\tilde c_P\cdot s}\Big)^p.
\end{equation*}
From the computation in Remark \ref{rem:gf-Lip}, it follows that
\begin{equation}\label{eq:gf-Lip}
  \frac{m_f(x,r)}{r}
  \leq 2\sup_{\rho(y,x)\leq r}\frac{\abs{f(y)-f(x)}}{r}
  \leq 2\sup_{0<\rho(y,x)\leq r}\frac{\abs{f(y)-f(x)}}{\rho(x,y)}
  \leq 2\Norm{f}{\Lip},
\end{equation}
so in particular the integrand above is non-negative. Hence, we may apply Fatou's lemma to obtain
\begin{equation*}
\begin{split}
  (2\Norm{f}{\Lip})^p-\Big(\frac{m_f(z,s)}{\tilde c_P\cdot s}\Big)^p
  &\geq\liminf_{r\to 0}\fint_{B(z,\tilde\lambda s)}\Big[(2\Norm{f}{\Lip})^p-\Big(\frac{m_f(x,r)}{r}\Big)^p\Big]\ud\mu(x) \\
  &\geq\fint_{B(z,\tilde\lambda s)}\liminf_{r\to 0}\Big[(2\Norm{f}{\Lip})^p-\Big(\frac{m_f(x,r)}{r}\Big)^p\Big]\ud\mu(x) \\
  &=\fint_{B(z,\tilde\lambda s)}\Big[(2\Norm{f}{\Lip})^p-\limsup_{r\to 0}\Big(\frac{m_f(x,r)}{r}\Big)^p\Big]\ud\mu(x).
\end{split}
\end{equation*}
Thus
\begin{equation*}
\begin{split}
  \Big(\frac{m_f(z,s)}{\tilde c_P\cdot s}\Big)^p
  &\leq \fint_{B(z,\tilde\lambda s)}\limsup_{r\to 0}\Big(\frac{m_f(x,r)}{r}\Big)^p\ud\mu(x) \\
  &\leq \fint_{B(z,\tilde\lambda s)}\limsup_{r\to 0}\Big(2\sup_{0<\rho(x,y)\leq r}\frac{\abs{f(x)-f(y)}}{\rho(x,y)}\Big)^p\ud\mu(x)
  \quad\text{(by \eqref{eq:gf-Lip})} \\
  &=\fint_{B(z,\tilde\lambda s)}(2\Lip f(x))^p\ud\mu(x).
\end{split}
\end{equation*}
This is exactly the form of the $(1,p)$-Poincar\'e inequality in Lemma \ref{lem:Poincares}\eqref{it:Lipf}, and completes the proof of \eqref{it:macroPL}$\Rightarrow$\eqref{it:usualP}.

Finally, if $X$ is complete, then Lemma \ref{lem:Poincares} guarantees that \eqref{it:usualP} is equivalent to the $(1,p)$-Poincar\'e inequality in the sense of Definition \ref{def:Poincare}.
\end{proof}

Proposition \ref{prop:newPoincare} is somewhat in the spirit of \cite[Theorem 4.1]{DCJS}, which shows the equivalence of
\begin{enumerate}[\rm(a)]
 \item\label{it:DCJSold} a usual $(1,p)$-Poincar\'e inequality, and 
 \item\label{it:DCJSnew}a variant, where the right-hand side is an abstract functional $a_f(B)$ satisfying a list of postulates given in \cite[Definition 3.2]{DCJS}.
\end{enumerate}
However, a closer look shows that the similarity is only superficial. In \cite[Theorem 4.1]{DCJS}, the implication \eqref{it:DCJSold}$\Rightarrow$\eqref{it:DCJSnew} consists of showing that, denoting by $\mathfrak g_f$ a minimal $p$-weak upper gradient of $f$, the functional 
\begin{equation*}
  a_f(B):=C\cdot r(B)\Big(\fint_B \mathfrak g_f^p\ud\mu\Big)^{\frac1p}
\end{equation*}
satisfies the abstract postulates, so their ``new'' Poincar\'e inequality is just a version of the usual one (essentially, the one of Lemma \ref{lem:Poincares}\eqref{it:uppergrad}). In the other direction, where the usual Poincar\'e inequality is deduced from an abstract one, a key property of the abstract inequality is \cite[Definition 3.2($V_p$)]{DCJS}, which looks a bit like our macroscopic Poincar\'e inequality of Theorem \ref{thm:macroPoincare}, but involves a family of balls of arbitrary radius on the right-hand side, instead of the balls $B(x,r)$ of fixed radius $r$ as in Theorem \ref{thm:macroPoincare}. Hence, in both directions, \cite[Theorem 4.1]{DCJS} and Proposition \ref{prop:newPoincare} are essentially different.

As another detail, we also note that \cite[Theorem 4.1]{DCJS} only deals with complete spaces, while the main part of Proposition \ref{prop:newPoincare} is independent of this assumption.

%
%


\end{document}